\documentclass[11pt]{amsart}

\usepackage[top=3.5cm, bottom=3.5cm, left= 3cm, right= 3cm, centering]{geometry}
\usepackage{xcolor}
\usepackage{esint}
\usepackage{amssymb}

\usepackage{graphicx}
\usepackage{stmaryrd}
\usepackage{mathtools}

\usepackage[toc]{appendix}

\usepackage[colorlinks=true, pdfstartview=FitV, linkcolor=blue, citecolor=blue, urlcolor=blue,pagebackref=false]{hyperref}
\usepackage{microtype}

\usepackage{bm}
\usepackage{scalerel} 
\usepackage{mathrsfs}
\usepackage[english]{babel}
\usepackage[font={footnotesize}]{caption}
\usepackage{enumitem}
\usepackage{comment}
\usepackage{scalerel,stackengine}

\usepackage[symbol]{footmisc}

\stackMath
\newcommand\reallywidehat[1]{
	\savestack{\tmpbox}{\stretchto{
			\scaleto{
				\scalerel*[\widthof{\ensuremath{#1}}]{\kern-.6pt\bigwedge\kern-.6pt}
				{\rule[-\textheight/2]{1ex}{\textheight}}
			}{\textheight}
		}{0.5ex}}
	\stackon[1pt]{#1}{\tmpbox}
}
\definecolor{darkgreen}{rgb}{0,0.5,0}
\definecolor{darkblue}{rgb}{0,0,0.7}
\definecolor{darkred}{rgb}{0.9,0.1,0.1}
\definecolor{lightblue}{rgb}{0,0.51,1}

\usepackage[utf8]{inputenc}
\usepackage{chngcntr}
\usepackage{apptools}

\newtheorem{theorem}{Theorem}[section]
\newtheorem{definition}[theorem]{Definition}
\newtheorem{remark}[theorem]{Remark}
\newtheorem{lemma}[theorem]{Lemma}
\newtheorem{proposition}[theorem]{Proposition}
\newtheorem{assumption}[theorem]{Assumption}

\numberwithin{equation}{section}
\numberwithin{theorem}{section}

\newcommand{\N}{\mathbb{N}}
\newcommand{\R}{\mathbb{R}}

\renewcommand{\subset}{\subseteq}

\renewcommand{\tilde}{\widetilde}

\allowdisplaybreaks

\newcommand{\Ker}{{\rm Ker}\,}

\def\d{\partial}

\newcommand{\vertiii}[1]{{\left\vert\kern-0.25ex\left\vert\kern-0.25ex\left\vert #1 
		\right\vert\kern-0.25ex\right\vert\kern-0.25ex\right\vert}} 

\begin{document}
	
	\author[E. Bocchi]{Edoardo Bocchi}
	\address[E. Bocchi]{Departamento de An\'alisis Matem\'atico \& Instituto de Matem\'aticas de la Universidad de Sevilla, Universidad de Sevilla, Avenida Reina Mercedes, 41012 Sevilla, Espa\~{n}a}
	\email{ebocchi@us.es}
	
	\author[J. He]{Jiao He}
	\address[J. He]{Université Paris-Saclay, CNRS, Laboratoire de mathématiques d’Orsay, 91405, Orsay, France.}
	\email{jiao.he@universite-paris-saclay.fr}

	\author[G. Vergara-Hermosilla]{Gastón Vergara-Hermosilla}
	\address[G. Vergara-Hermosilla]{Institut de Math\'ematiques de Bordeaux, Universit\'e de Bordeaux, 351 cours de la Liberation, 33405, Talence, France.}
\email{gaston.v-h@outlook.com}

	\title[Well-posedness of an OWC in shallow water]{Well-posedness of an oscillating water column in shallow water with time-dependent air pressure}
	
\title[Well-posedness of an OWC in shallow water]{Well-posedness of a nonlinear shallow water model for an oscillating water column with time-dependent air pressure}
	
	
\begin{abstract}
We propose in this paper a new nonlinear mathematical model of an oscillating water column (OWC). The one-dimensional shallow water equations in the presence of this device is reformulated as a transmission problem related to the interaction between waves and a fixed partially-immersed structure. By imposing the conservation of the total fluid-OWC energy in the non-damped scenario, we are able to derive a transmission condition that involves a time-dependent air pressure inside the chamber of the device, instead of a constant atmospheric pressure as in \cite{bocchihevergara2021}.
We then show that the  transmission problem can be reduced to a quasilinear hyperbolic initial boundary value problem with a semi-linear boundary condition determined by an ODE depending on the trace of the solution to the PDE at the boundary.
Local well-posedness for general problems of this type is established via an iterative scheme by using linear estimates for the PDE and nonlinear estimates for the ODE. \\

\end{abstract}
\maketitle

\section{Introduction}
\subsection{General settings.}
This article is devoted to the modelling and the mathematical analysis of a particular wave energy converter (WEC) called oscillating water column (OWC). In this device, incoming waves arrive from the offshore, collide against a partially-immersed fixed structure. The incident wave rebounds but part of the water passes below the structure and enters a partially-closed chamber, whose boundaries are the partially-immersed structure at the left, a solid wall at the right and a solid wall with a hole at the top, see Figure \ref{OWC}. The water volume inside the chamber increases and compresses air at the upper end of the chamber, forcing it through the hole where a turbine is installed and creates electrical energy. Viceversa, when the water volume decreases, the air outside the chamber enters, activates the turbine and increases the air volume inside the chamber.  The name OWC comes from the fact that it behaves like an oscillating liquid piston, a water column, that compresses air inside the chamber. Therefore, this device allows to convert the energy (both kinetic and potential) associated with a moving wave into useful energy. 
For more details on the transformation of wave energy to electric energy in this type of WEC we refer to 
\cite{pecher2017handbook}. OWCs are one example of a large variety of WECs that can be found in hydrodynamical engineering. For their classification and description, we refer the interested readers to
\cite{babarit2018}.\\
Among all these devices, floating structures and their interaction with water waves have been particularly studied in the last years. In \cite{lannes2017floating} Lannes derived a model for the interaction between waves and floating structures taking into account nonlinear effects, which have been neglected in previous analytical approaches in the literature (see for instance \cite{john1949,john1950} where floating structures first were modelled). He derived the model in the general multidimensional case considering a depth-averaged formulation of the water waves equations and then the shallow water asymptotic models for the fluid motion given by the nonlinear shallow water equations and the Boussinesq equations. In \cite{IguLan21} Iguchi and Lannes proved the local well-posedness of the one-dimensional nonlinear shallow water equations in the presence of a freely moving floating structure with non-vertical side-walls. In \cite{bocchi2020floating}  Bocchi showed 
the local well-posedness
of the nonlinear shallow water equations in the two-dimensional axisymmetric without swirl case for a floating object moving only vertically and with vertical side-walls. In \cite{bresch2019waves} Bresch,  Lannes and Métivier considered the case when the structure is fixed with vertical walls and the fluid equations are governed by the one-dimensional Boussinesq equations. Local well-posedness was obtained in the same time scale as in the absence of an object, that is $O(\varepsilon^{-1})$ where $\varepsilon$ is the nonlinearity parameter. Recently, Beck and Lannes in \cite{lannesbeck2021boussinesq} extended the previous analysis to the case of a floating structure with vertical or non-vertical side-walls having only a vertical motion, for which a shorter time scale $O(\varepsilon^{-1/2})$ is obtained. In \cite{mai-sm-taka-tuc19} Maity, San Mart{\'\i}n, Takahashi and Tucsnak considered one-dimensional viscous shallow water equations and a solid with vertical side-walls moving vertically. In this viscous case, they showed local well-posedness for every initial data and 
global one if the initial data are close to an equilibrium state. Furthermore, a particular configuration has been investigated, called the return to equilibrium, where the floating structure is dropped from a non-equilibrium position with zero initial velocity into the fluid initially at rest and let evolve towards its equilibrium state. This problem can be easily done experimentally in wave tanks and is used to determine important characteristics of floating objects. Engineers assume that the solid motion is governed by a linear integro-differential equation, the Cummins equation, that was empirically derived by Cummins in \cite{cummins1962impulse} and the experimental data coming from this configuration are used to determine the coefficients of this equation. A nonlinear Cummins equation in the one-dimensional case was derived by Lannes in \cite{lannes2017floating} and a nonlinear integro-differential Cummins equation was derived in the two dimensional axisymmetric without swirl case by Bocchi in \cite{bocchi2020onthereturn}.
Recently, Beck and Lannes in \cite{lannesbeck2021boussinesq} 
derived in the one-dimensional case an abstract  Cummins-type equation that takes an explicit simple form in the nonlinear non-dispersive and the linear dispersive cases. More recently,   Vergara-Hermosilla,  Matignon and Tucsnak in  \cite{mati-tuc-vergara21} derived explicitly the asymptotic behavior of a Cummins-type equation including viscous effect in the one-dimensional case. 

In the last decades oscillating water columns have been widely investigated both experimentally and numerically in the hydrodynamical engineering literature for the sake of understanding how to increase the performance of these wave energy converters in order to facilitate a real installation. For instance, we refer to \cite{Dimakopoulos-Cooker-Bruce2017, evans1995, fanhelgato16air, lopez2015, reza2013, reza2018, reza2017}  and references therein. 
All these works were essentially based on the linear water wave theory introduced
by Evans
in \cite{evans1978, evans1982}, in which the wave motion is assumed time-harmonic. Motivated by the lack of a nonlinear analysis for this type of wave energy converter, we modelled and simulated an OWC in a first paper \cite{bocchihevergara2021} taking into account the nonlinear effects following the series of works in the case of floating structures presented before. As a first and simpler approach, a constant air pressure was considered inside the chamber, although it does not seem realistic since the variations of the fluid volume cause variations of the air volume inside the chamber.
Moreover, inspired by \cite{reza2013} we considered in the model of \cite{bocchihevergara2021} a discontinuous topography by adding a step in the sea bottom in front of the device. Recently, the exact boundary controllability of that simplified OWC model was treated by Vergara-Hermosilla, Leugering and Wang in \cite{vergara-leugering-wang2021}.

This article is a direct continuation of \cite{bocchihevergara2021} and its aim is twofold:
\begin{enumerate}
    \item \label{goal1}
    derive a nonlinear model that describes the interaction between waves and the OWC by taking into account time-variations of the air pressure inside the chamber;\\[1pt]
     \item \label{goal2} 
     establish a local well-posedness result for the transmission problem across the structure side-walls in the Sobolev setting.
\end{enumerate}
Since the interest of this new work lies only in the wave-structure interaction of the OWC, we do not consider neither the open sea situation nor the step in front of the device, whose rigourous mathematical analysis has already been treated in \cite[Section 6.1]{IguLan21}. Indeed, we work with a 
bounded fluid domain with a flat bottom.

\subsection{Main notations} 
The configuration of the wave energy device considering is presented on Figure \ref{OWC}.
\begin{figure}
\centering  
\includegraphics[width=0.6\textwidth]{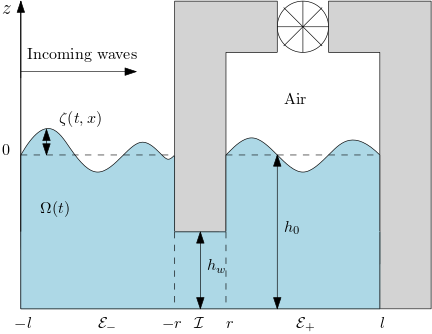}
\caption{Configuration of the oscillating water column device.}
\label{OWC}
\end{figure}
Let us give several notations that will be used throughout the paper.
\subsubsection*{Notation of domains}
\begin{itemize}
    \item[-] We divide the spatial domain $(-l, l)$ into the interior domain and the exterior domain given respectively by
    $$
    \mathcal{I}:= (-r,r)\quad \mbox{ and }\quad
    \mathcal{E}=\mathcal{E}_-\cup \mathcal{E}_+ :=(-l,-r)\cup (r,l).
    $$

    \item[-] We write the time-space domain $\Omega_T:=(0,T)\times \mathcal{E}_+.$
    \end{itemize}
    
\subsubsection*{Functions and constants}
\begin{itemize}
    \item[-] $\zeta(t,x)$ \hspace{0.44cm} Surface elevation 
    \item[-] $\zeta_w$ \hspace{1.06cm} Bottom of the partially-immersed structure
    \item[-] $h(t,x)$\hspace{0.57cm} Fluid height
    \item[-] $h_w$\hspace{1.15cm} Fluid height under the structure
    \item[-] $q(t,x)$\hspace{0.61cm} Horizontal discharge 
    \item[-] $q_i(t)$\hspace{0.9cm} Horizontal discharge in $\mathcal{I}$
    \item[-] $ \underline{P}(t,x)$\hspace{0.51cm} Surface pressure of the fluid
    \item[-] $P_\mathrm{air}(t,x)$\hspace{0.16cm} Air pressure
    \item[-] $P_\mathrm{ch}(t)$\hspace{0.6cm}  Time-dependent variation of  $ P_\mathrm{air}$ inside the OWC chamber
    \item[-] $ P_\mathrm{atm}$\hspace{0.8cm}  Constant atmospheric pressure
    \item[-] $h_0$\hspace{1.23cm} Fluid height at rest in   $\mathcal{E}$
    \item[-] $\dot f$ \hspace{1.25cm} Time derivative of a function  $f$
depending only on $t$
  \item[-] $f^{
(k)}$ \hspace{0.8cm} $k$-th time derivative of a function  $f$
depending only on $t$
  \item[-] $ f(0)$ \hspace{0.77cm} Evaluation of $f$ at $t=0$
\item[-]  $C(\cdot)$ \hspace{0.8cm}  Generic function with number of arguments that may differ from line to line
\end{itemize}

\subsubsection*{Spaces and norms}
\begin{itemize}
     \item[-] For $m\in \N$ and $X= \mathcal{E}_+$ or $\Omega_T$, we denote the norms of $H^m(X)$ and $W^{m,\infty}(X)$ respectively by $\|\cdot\|_{H^m(X)}$ and $\|\cdot\|_{W^{m,\infty}(X)}$.
     \item[-] For $m\in \N$, $\mathbb{W}^m(T)$ is the 
      function space
     defined by
     \begin{equation}\label{defWm}
\mathbb{W}^{m}(T):=\bigcap_{j=0}^{m} C^{j}\left([0, T] ; H^{m-j}
(\mathcal{E}_+)
\right)
\end{equation}
endowed with the norm 
$$
\|u\|_{\mathbb{W}^m(T)}:=\sup _{t \in[0, T]}\vertiii{u(t)}_{m}
, 
\quad \mbox{where}\quad 
\vertiii{u(t)}_{m}=\sum_{k=0}^{m}\|\partial_{t}^{k} u(t,\cdot)\|_{H^{m-k}
  (\mathcal{E}_+)
 }.
$$
Note that $H^{m+1} (\Omega_T) \subsetneq  \mathbb{W}^m (T) \subsetneq H^m (\Omega_T).$
\item[-] For $m\in \N$, we denote the norms of   $H^m(0,T)$ and $W^{m,\infty}(0,T)$ respectively by $|\cdot|_{H^m(0,T)}$ and $|\cdot|_{W^{m,\infty}(0,T)}$.
\item[-]  The trace norm $|u_{|_{x=r}}|_{m,T}$ is defined by
$$
|u_{\left.\right|_{x=r}}|_{m, T}^{2}:=\sum_{k=0}^{m}|(\partial_{x}^{k} u)_{\left.\right|_{x=r}}|_{H^{m-k}(0, T)}^{2}=\sum_{|\alpha| \leq m}|\left(\partial^\alpha u\right)_{\left.\right|_{x=r}}|_{L^{2}(0, T)}^{2},
$$ with $\partial^\alpha:=\partial _t^{\alpha_1} \partial_x^{\alpha_2}$  for $\alpha=(\alpha_1, \alpha_2)$ and $|\alpha|=\alpha_1 + \alpha_2.$ 
Moreover, we use the notation
$$|u_{|_{x=r,l}}|_{m,T}:=|u_{|_{x=r}}|_{m,T}+ |u_{|_{x=l}}|_{m,T}.$$
\item[-] Given $X$ a generic function space
with norm $\| \cdot\|_X$, the compact notation $C(\|u, v\|_X)$ denotes $C(\|u\|_X, \|v\|_X )$.

\end{itemize}

\subsection{Main techniques and novelties.}
We summarize here the equations studied in this article, the techniques used in the analysis and the main results obtained. 

\begin{itemize}
\item  
In our model, the fluid equations are given by 1d nonlinear shallow water equations with a fixed partially-immersed structure and the air pressure is considered to be time-dependent inside the chamber of the device. More precisely, we obtain the following transmission problem related to the fixed partially-immersed structure with vertical-side walls at $x=\pm r$:
\begin{equation}\label{sidewalls-pb}
\begin{cases} 
\partial_t\zeta + \partial_x q=0 
 \\[6pt]
\partial_t q + \partial_x \left(\dfrac{q^2}{h_0+\zeta} \right)+g (h_0+\zeta)\partial_x \zeta=0\end{cases}
\mbox{ in } \quad (0,T)\times\mathcal{E}
\end{equation} 
with  boundary conditions
\begin{equation*}
    \zeta_{|_{x=-l}}=\zeta_{\rm ent}, \qquad q_{|_{x=l}}=0,
\end{equation*}
and transmission conditions
\begin{equation}
\label{trans-sidewalls}
\llbracket q \rrbracket=0, \quad\quad \langle q\rangle =q_i,
\end{equation} 
where $q_i$, $P_\mathrm{ch}$ satisfy
\begin{equation}
\begin{cases}
\displaystyle\frac{d q_i}{dt} = -\frac{1}{\alpha}\left\llbracket g \zeta + \frac{q^2}{2(h_0+\zeta)^2}\right\rrbracket -\frac{1}{\alpha\rho}P_\mathrm{ch},
\\[10pt]
\displaystyle\frac{d P_\mathrm{ch}}{dt} =- \gamma_1 P_\mathrm{ch}+ \gamma_2q_i.
\end{cases}
\end{equation}
The initial conditions are  
\begin{equation}\label{initial}
 \zeta(0,x)=\zeta_0(x), \quad q(0,x)=q_0(x)  \quad \mbox{in}\quad  \mathcal{E}, \qquad 
 \text{and}
 \qquad
q_i(0)=q_{i,0}, \quad P_\mathrm{ch}(0)=P_{\mathrm{ch},0}. 
\end{equation}
The boundary datum $\zeta_{\rm ent}$ is a given time-dependent entry function, $\gamma_1,\gamma_2$ are some positive constants, $\rho$ is the constant fluid density and $\alpha=\frac{2r}{h_w}$ with $h_w = h_0+\zeta_w$. The notations $\llbracket q \rrbracket$ and $\langle q\rangle$ denote respectively the jump and the average of $q$ at $x=\pm r$, namely 
$$\llbracket q \rrbracket:=q_{|_{x=r}}- q_{|_{x=-r}} \quad \mbox{ and }\quad  \langle q\rangle:=\tfrac{1}{2}\left(q_{|_{x=-r}}+q_{|_{x=r}}\right).$$ 
The first novelty is that, to the authors' knowledge, this is the first nonlinear model for the interaction between shallow water waves and an OWC involving a time-dependent air pressure inside the chamber of the device. Adapting the argument used in our previous work \cite{bocchihevergara2021}, we obtain a transmission condition imposing conservation of the total fluid-OWC energy in the non-damped scenario (see Subsection \ref{sec3_deri_second_trans_cond}). The OWC energy is mathematically derived from the structure of the ODE governing the dynamics of the air pressure perturbation inside the chamber.
This derivation improves and generalises the previous nonlinear model derived in \cite{bocchihevergara2021}, as one can recover the same transmission condition in the case of a constant air pressure inside the chamber.\\

\item The second contribution of this article is the following local well-posedness result for the previous transmission problem in the Sobolev setting.

\begin{theorem}\label{theomainresult}
Let $m\geq 2$ be an integer and $(\zeta_0, q_0)\in H^m(\mathcal{E})$ be such that Assumption \ref{assOWC} holds. Suppose that $(\zeta_0, q_0)$,  $(q_{i,0}, P_\mathrm{ch,0})\in\R^2$ and $\zeta_{\rm ent}\in H^m(0,T)$ satisfy the compatibility conditions up to order $m-1$. Then there exists $0<T_1\leq T$ and a unique solution $(\zeta, q, q_i, P_\mathrm{ch})$  to
\eqref{sidewalls-pb}-\eqref{initial} 
 with $(\zeta, q)\in \mathbb{W}^m(T_1)$ and $(q_i,P_\mathrm{ch})\in H^{m+1}(0,T_1)$,  where $\mathbb{W}^m(T_1)$ denotes the same space as in \eqref{defWm} but defined in the spatial domain $\mathcal{E}$.   Moreover, $|(\zeta,q)_{|_{x=\pm r, \pm l}}|_{m,T_1}$ is finite.
\end{theorem}

To the best of our knowledge, this represents the first well-posedness result in the Sobolev setting of a nonlinear model for the interaction between waves and the OWC. It is achieved by reformulating \eqref{sidewalls-pb}-\eqref{trans-sidewalls} as a one-dimensional $4\times 4$ hyperbolic quasilinear initial boundary value problem with a semilinear boundary condition, i.e.
\begin{equation}
		\begin{aligned}
		\begin{cases}
		\partial_t u + \mathcal{A}(u)\partial_x u = 0 &\quad\mbox{in}\quad  (0,T)\times 
	 \mathcal{E}_+
		,\\[5pt]
		u(0)=u_0(x) &\quad\mbox{on}\quad \mathcal{E}_+,\\[5pt]
		\mathcal{M}_r u_{
		|_{x=r}
		}=V(G(t)) &\quad\mbox{on}\quad  (0,T),\\[5pt]
		\mathcal{M}_l u_{|_{x=l}}=g(t) &\quad\mbox{on}\quad  (0,T),
		\end{cases}
		\end{aligned}
		\label{PDE}
		\end{equation} 
	where $u,u_0$ are $\mathbb{R}^4$-valued functions, $\mathcal{A}(u)$, $\mathcal{M}_r$ and $\mathcal{M}_l$ are respectively $4\times 4$, $2\times4$ and $2\times4$ real-valued matrices, $V $ and $g$ are $\mathbb{R}^2$-valued functions and $G$ is a $\mathbb{R}^2$-valued function satisfying the following equation	
\begin{equation}
		\begin{cases}
		\dot{G}= \Theta(G, u_{
		 |_{x=r}
		} ) \quad \mbox{in} \quad (0,T),\\[5pt]
		G(0)=G_0,
		\end{cases}
		\label{ODE}
\end{equation}with $\Theta: \mathbb{R}^2\times \mathbb{R}^4 \rightarrow \mathbb{R}^2$. See Section \ref{Section 3: derivation_model} for their explicit expressions.
We take advantage of the one-dimensional setting to construct an explicit Kreiss symmetrizer. This is done by adding two weight functions, one larger enough than the other one at $x=r$ and viceversa at $x=l$, in the expression of the symmetrizers in \cite{bocchi2020floating,IguLan21}. This new adjustment permits to handle the two boundaries of the domain (contrarily to only one boundary in the half-line case in \cite{bocchi2020floating,IguLan21}). Then, the assumption of an equivalent version of the so-called uniform Kreiss-Lopatinski\u{i} condition makes the two boundary conditions dissipative.  Roughly speaking, this property allows us to control the traces $u_{|_{x=r,l}}$ at the same regularity as $u$, without loss of derivatives. 
We notice that the minimal regularity index obtained in Theorem \ref{theomainresult}, that is $m=2$, corresponds to the standard minimal regularity integer index $m>d/2 +1$ for one-dimensional quasilinear initial value problems.\\
The proof is based on the study of the linearized ``PDE system" and an iterative scheme for the coupled ``PDE-ODE system". 
As usually done for initial boundary value problems, we prove the boundedness of a sequence of approximated solutions in some ``high norm" and its convergence in some ``low norm" (see  \cite{benzoni-serre2007, coron_2007_book}) by using linear high order Sobolev estimates for the PDE together with nonlinear high order Sobolev estimates for the ODE. 
While PDE’s estimates were already derived in \cite{IguLan21}, the difficulty of our proof arises from the fact that the boundary data is not given but determined by an evolution equation depending on the trace of the solution at the boundary.
To handle this, we derive estimates for the iterative ODE involving the norm of $u_{|_{x=r}}$ controlled via the linear estimates and, moreover, with a time factor that goes to zero as the existence-time $T$ goes to zero.
Indeed, this time-dependence together with the choice of a small $T$, is crucial to close the iterative argument that gives both boundedness and convergence.
The limit of the sequence is then the solution $(u,G)$ to \eqref{PDE}-\eqref{ODE} and its uniqueness and regularity follow by standard arguments.
\end{itemize}
\subsection{Organization of the paper}
The outline of the article is as follows. We present in Section \ref{Section 2} the nonlinear mathematical model of an oscillating water column in the shallow water regime. In Subsection \ref{section2: fluid_eqs} we first introduce the different domains involved in the model and  present the one-dimensional nonlinear shallow water equations in the presence of a partially-immersed structure. 
After showing the duality property of constraints and unknowns, we split the equations into two different systems corresponding respectively to the region where the fluid surface is free  and the region under the structure where the surface is constrained. Moreover, boundary conditions are given to complete the model. Subsection \ref{sec2_air_pres_dyn} is devoted to the air pressure dynamics. We assume that the air pressure is equal to the constant atmospheric pressure outside the chamber and we consider it as a time-dependent variation of the atmospheric pressure inside the chamber. We explicitly give the evolution equation of the air pressure variation and rewrite it in terms of the horizontal discharge under the partially-immersed structure.\\
In Section \ref{Section 3: derivation_model} we reformulate the model as a transmission problem. In Subsection \ref{sec3_trans_strucure} we distinguish the equations in the region before the structure, and after the structure, which is the domain inside the chamber. The continuity of the horizontal discharge at the side-walls gives one transmission condition. However, due to the lack of continuity for the surface elevation at the side-walls, one additional condition is necessary to close the system and guarantee the well-posedness of this problem. Therefore in Subsection \ref{sec3_deri_second_trans_cond} we derive a second transmission condition imposing the conservation of the total fluid-OWC
energy in the non-damped scenario. The new transmission condition takes into account the time-dependent variation of the air pressure inside the chamber. 
We show in Subsection \ref{sec3_transmiss_pron_as_IBVP} that the transmission problem can be recast as a $4\times 4$ initial boundary value problem with a semilinear boundary condition.
In Section \ref{KreissHypIBVP} we investigate the well-posedness for general quasilinear hyperbolic IBVPs with a semilinear boundary condition. In Subsection \ref{subsection_linearsystem}, we first present the well-posedness theory of Kreiss-symmetrizable linear hyperbolic IBVP with variable coefficients and given boundary data. This was treated in \cite{IguLan21} in the half-line case and here we adapt it to the bounded interval case. 
More precisely, we construct a \textit{Kreiss symmetrizer} adding two weights functions in the expression of the symmetrizers of \cite{bocchi2020floating,IguLan21} in order to handle both boundaries of the domain. Afterwards, we introduce the notions of \textit{uniform Kreiss-Lopatinski\u{i}} and compatibility conditions, which are necessary for higher order \textit{a priori} estimates and the well-posedness of the linear IBVP stated in Theorem \ref{thm_Iguchi_lannes}. 
In Subsection \ref{subsection_nonlinear} we present some Moser-type nonlinear estimates that we repetitively use in the proof of the well-posedness theorem for the quasilinear IBVP. In Subsection \ref{subsection_ODE}, we establish some required nonlinear estimates for the ODE that determines the boundary condition in the IBVP. Using linear estimates for PDE and nonlinear estimates for ODE, in Subsection \ref{subsection_IBVP} we construct a solution to the quasilinear hyperbolic IBVP with a semilinear boundary condition by an iterative argument. In fact, the obtained solution is the limit of the sequence of approximated solutions to the coupled PDE-ODE system. 
In Subsection \ref{subsection_well-posedness} the well-posedness of the original problem is finally obtained as an application of the general theory.

\section{Derivation of the model}\label{Section 2}
\subsection{Fluid equations}\label{section2: fluid_eqs}
We consider an incompressible, irrotational, inviscid and homogeneous fluid that interacts  with an on-shore oscillating water column device in a shallow water regime. This means that characteristic fluid height is small with respect to the characteristic horizontal scale in the longitudinal direction.  Let us denote by $\zeta(t,x)$ the surface elevation, which is assumed to be a graph, and by $-h_0$ (with $h_0 >0$) the parametrization of the flat bottom. The two-dimensional fluid domain is $$\Omega(t)=\{(x,z)\in (-l, l)\times\R \ : -h_0<z< \zeta(t,x)\}.$$ 
The partially-immersed structure is centered at $x=0$, with horizontal length $2r$ and vertical walls located at $x=\pm r$. Its presence permits to divide the horizontal projection of the fluid domain into two domains: the \textit{exterior} domain $(-l, -r)\cup (r , l)$, where the 
water surface
is not in contact with the structure, and the \textit{interior} domain $(-r,r)$, where the contact occurs. We denote them by $\mathcal{E}$ and $\mathcal{I}$, respectively. Furthermore, later in the analysis we will need to distinguish the part of $\mathcal{E}$ outside the chamber and inside the chamber. Hence we denote by $\mathcal{E}_-$ and $\mathcal{E}_+$ the subsets $(-l,-r)$ and $(r,l)$, respectively.\\
The horizontal discharge $q(t,x)$ is defined by
\begin{equation*}
q(t,x) = 
\displaystyle\int_{-h_0}^{\zeta(t,x)} u(t,x,z)dz  \quad \mbox{for} \quad (t,x)\in(0,T)\times (-l,l),
\end{equation*}   
where $u(t,x,z)$ is the horizontal component of the fluid velocity. It follows that $q=h\overline{u}$ where $\overline{u}(t,x)$ is the vertically averaged horizontal fluid velocity and $h(t,x)=h_0 + \zeta(t,x)$ is the fluid height. After integrating over the fluid height the horizontal component of the free surface Euler equations, adimensionalizing the equations and truncating at precision $O(\mu)$, where $\mu$ is the shallowness parameter, one can obtain the nonlinear shallow water equations in the presence of a structure. We refer to \cite{lannes2017floating, lannes2020modeling} for the derivation of the equations in the  multi-dimensional case and \cite{bocchi2020floating} in the two-dimensional axisymmetric with no swirl case. 
Here we consider the one-dimensional nonlinear shallow water equations in the presence of a partially-immersed structure:
\begin{equation}\label{1dshallow-structure}
    \begin{cases}
    \partial_t\zeta + \partial_x q=0,\\[5pt]
    \partial_t q + \partial_x \left(\dfrac{q^2}{h} \right)+g h\partial_x \zeta=-\dfrac{h}{\rho}\partial_x \underline{P},
    \end{cases}\quad \mbox{ in } \quad (0,T)\times(-l,l),
\end{equation}
where $\underline{P}(t,x)$ is the surface pressure of the fluid, $g$ is the gravitational constant and $\rho$ is the constant fluid density.
In this paper viscous effects are not taken into account. However, some numerical-based models considered in the wave-energy community (for instance \cite{wang2018nonlinear}) showed that viscosity effects should be included to have a good agreement with experimental data. We also refer to \cite{mai-sm-taka-tuc19} for more about viscous shallow water model for a floating solid where the authors were able to obtain a global well-posedness result due to the viscosity term.
Moreover, we do not include capillary effects since in the characteristic scale of the problem they are negligible.
Indeed, we assume continuity of the surface pressure with the air pressure outside the fluid domain. In general, the air pressure is taken equal to the constant (both in time and space) atmospheric pressure. 
In a first and simpler approach, the authors modelled the oscillating water column device in \cite{bocchihevergara2021} with a constant air pressure, both outside and inside the chamber. A novelty of this work is that we consider an air pressure function which is not constant through all the domain. Indeed, while outside the chamber it is reasonable to consider a constant air pressure, inside the chamber the motion of the waves produce variations of the air pressure and this fact must be taken into consideration to describe more precisely the behaviour of a wave energy converter of this type.\\
Let us now talk about the partially-immersed structure. We assume that the bottom of the structure can be parametrized as graph of a function $\zeta_w$ and for the sake of simplicity we consider a solid with a flat bottom, yielding $\zeta_w=\zeta_w(t)$. We remark that the same theory holds in the case of objects with non-flat bottom. The fact that in an oscillating water column device the partially-immersed structure is fixed implies that $\zeta_w$ is a constant of the problem both in space and time. Dealing with floating structures leads to consider a time-dependent function $\zeta_w$ related to the velocity of the moving object (see \cite{bocchi2020floating,lannes2017floating} for nonlinear shallow water equations, \cite{lannesbeck2021boussinesq} for Boussinesq equations).\\

\paragraph{\textbf{Constraints and unknowns}}The interaction between floating or fixed structures and water waves, inherits a duality property. On the one hand, in the exterior domain, the surface pressure is constrained to be equal the air pressure while the surface elevation is free, i.e.
\begin{equation}
    \label{cons-unk_ext}\begin{cases}
    \underline{P}(t,x)= P_\mathrm{air}(t,x),\\
    \zeta(t,x) \mbox{ is unknown},
    \end{cases}
    \quad \mbox{for} \quad (t,x)\in(0,T)\times \mathcal{E},
\end{equation}
where $P_\mathrm{air}(t,x)$ is the known air pressure function. On the other hand, in the interior domain, the surface elevation matches the bottom of the solid while the surface pressure is free, i.e.
\begin{equation}
    \label{cons-unk_int}
    \begin{cases}
    \zeta(t,x)=\zeta_w,\\
    \underline{P}(t,x) \mbox{ is unknown},
    \end{cases}
    \quad \mbox{for} \quad (t,x)\in (0,T)\times \mathcal{I}.
\end{equation}
It has been shown in \cite{lannes2017floating} that the pressure $\underline{P}$ in the interior domain can be seen as a Lagrange multiplier associated with the contact constraint $ \zeta(t,x)=\zeta_w$ (it holds also for the water waves equations in the presence of a floating structure). Injecting \eqref{cons-unk_ext}-\eqref{cons-unk_int} into \eqref{1dshallow-structure}, we obtain the following two systems

\begin{equation}\label{1dshallow-ext-as}
    \begin{cases}
    \partial_t\zeta + \partial_x q=0,\\[5pt]
    \partial_t q + \partial_x \left(\dfrac{q^2}{h_0+\zeta} \right)+g (h_0+\zeta)\partial_x \zeta=-\dfrac{h_0+\zeta}{\rho}\partial_x P_\mathrm{air},
    \end{cases}\quad \mbox{ in } \quad (0,T)\times\mathcal{E},
\end{equation}
and
\begin{equation}\label{1dshallow-int}
    \begin{cases}
    q=q_i(t),\\[5pt]
    \dfrac{d q_i}{dt}=-\dfrac{h_w}{\rho}\partial_x \underline{P}, 
    \end{cases}\quad \mbox{ in } \quad (0,T)\times\mathcal{I}, 
\end{equation}
where $q_i$ is a time-dependent function that coincides with the horizontal discharge in the interior domain. 
Notice that the first equation in \eqref{1dshallow-int} comes from the continuity equation $\partial_t \zeta + \partial_x q =0$ together with constraint \eqref{cons-unk_int} in the interior domain.

\paragraph{\textbf{Boundary conditions}}
Let us discuss here the boundary conditions that couple with \eqref{1dshallow-ext-as}-\eqref{1dshallow-int}. 
As in \cite{bocchihevergara2021,lannes2020generating} we deal with a left boundary at $x=-l$ and the boundary condition reads
$$\zeta_{|_{x=-l}}=\zeta_{\rm ent},$$
where $\zeta_{\rm ent}= \zeta_{\rm ent}(t)$ is a given time-dependent entry function. This is necessary when dealing with numerical applications and $\zeta_{\rm ent}$ can be determined from experimental data. Indeed, during experiments in wave tanks it is usual to create waves with a lateral piston that permits to know the exact entry value of the surface elevation at any given time. Moreover,  in \cite{lannes2020generating} the authors showed that the knowledge of the entry value of the surface elevation allows to get the entry value of the horizontal discharge using the existence of Riemann invariants for the 1d nonlinear shallow water equations.\\
At the vertical walls of the partially-immersed structure we consider the slip condition for the fluid velocity. Moreover, since the fluid is irrotational, we know that the fluid velocity is continuous in the interior of $\Omega(t)$ from the elliptic regularity of the velocity potential. Combining these two facts, the continuity of the horizontal discharge at the walls follows (see more details in \cite{lannes2017floating}). Of course, since the structure have vertical walls, the continuity of the surface elevation at the solid walls and of the surface pressure fails (this would not be the case for instance in the case of a boat, see \cite{IguLan21,lannes2017floating}). Thus, we have
\begin{equation}
    \label{qcon-walls}
    q_{|_{x=(\pm r)^+}}=q_{|_{x=(\pm r)^-}}.
\end{equation}
We will see in the next section how to supply the lack of continuity for  both the pressure and the surface elevation at the structure walls and derive a condition which will close the system. Finally,  at the end of the chamber we consider a solid wall condition, that is
\begin{equation}
\label{wallcondition}
    q_{|_{x=l}}=0.
\end{equation}

\subsection{Air pressure dynamics}\label{sec2_air_pres_dyn}
In this subsection we focus on the air pressure, which is not in general a constant function. In particular, we distinguish the cases of the air outside the chamber and inside the chamber. On the one hand, in $\mathcal{E}_-$ the variations of the air pressure are negligible and it can be considered equal to the constant atmospheric 
pressure, 
\textit{i.e.}
\begin{equation}
    \label{constant-air-pressure}
    P_\mathrm{air}(t,x)= P_\mathrm{atm}\quad \mbox{for} \quad (t,x)\in(0,T)\times\mathcal{E}_-.
\end{equation}
On the other hand, in $\mathcal{E}_+$, where the air is partially trapped inside the chamber and pushed by the waves motion, a constant air pressure is no more realistic. 
We can reasonably assume that the air pressure inside the chamber is uniform in space. 
Therefore we deal with a time-dependent air pressure function and in particular we write it as a variation of the atmospheric pressure, \textit{i.e.}
\begin{equation}
    \label{time-dep-air-pressure}
    P_\mathrm{air}(t,x)= P_\mathrm{atm} + P_\mathrm{ch}(t) \quad \mbox{for} \quad (t,x)\in(0,T)\times \mathcal{E}_+,
\end{equation}
where $P_\mathrm{ch}(t)$ is the time-dependent variation. With this type of hypothesis on the air pressure inside the chamber, it is possible to find in ocean engineering literature an evolution equation governing the dynamics of the pressure variation $P_\mathrm{ch}(t)$. For instance, we refer to \cite{Dimakopoulos-Cooker-Bruce2017,fanhelgato16air}. It is derived for oscillating water column with Wells turbines \cite{raghunathan1995wells}, for which the relation between the pressure drop and the velocity of the air in the resistance layer is linear. Assuming this characteristics of the device, we have that $P_\mathrm{ch}$ satisfies the following linear ODE:
\begin{equation}
    \label{Pch-dynamics}
    \frac{d P_\mathrm{ch}}{dt} + \frac{\gamma P_\mathrm{atm}}{h_\mathrm{ch} K} P_\mathrm{ch}= \frac{\gamma P_\mathrm{atm}}{h_\mathrm{ch}}\frac{d\overline{\zeta}}{dt},
\end{equation}
where $\gamma$ is the polytropic expansion index of the air ($\gamma=1.4$), $h_\mathrm{ch}$ is the height of the chamber and $K$ is a resistance parameter. 
Despite these known parameters of the device, the spatially averaged free surface elevation $\overline{\zeta}$ over $\mathcal{E}_+$ remains unknown. In general in ocean engineering and marine energy literature, authors determine this value from experimental data calculated by gauges located inside the chamber. In our analytic approach, we rewrite it in terms of the horizontal discharge at the entrance of the chamber, that is at $x=r^+$. Indeed, using the continuity equation in \eqref{1dshallow-ext-as} we have 
\begin{equation*}
    \frac{d\overline{\zeta}}{dt}=\frac{d}{dt}\left(\frac{1}{|\mathcal{E}_+|}\int_{\mathcal{E}_+} \zeta(t,x)dx\right)= \frac{1}{|\mathcal{E}_+|}\int_{\mathcal{E}_+} \partial_t \zeta(t,x)dx = \frac{q_{|_{x=r^+}}}{|\mathcal{E}_+|}=\frac{q_i}{|\mathcal{E}_+|},
\end{equation*} 
where in the last two equalities we have used the wall condition \eqref{wallcondition} and the continuity condition \eqref{qcon-walls} for the horizontal discharge. Therefore \eqref{Pch-dynamics} reads
\begin{equation}\label{Pchamber-dynamics}
     \frac{d P_\mathrm{ch}}{dt} +\gamma_1 P_\mathrm{ch} = \gamma_2 q_i,
\end{equation}
where $\gamma_1$ and $\gamma_2$ are constants depending on the device parameters as in \eqref{Pch-dynamics}. Note that $\gamma_2 > 0$ ensures transmission while $\gamma_1 > 0$ tends to zero as the height of the chamber or some resistance of the device increases. For later purpose,  let us define the \textit{non-damped} scenario when $\gamma_1$ is negligible.
The previous equation \eqref{Pchamber-dynamics} shows that the dynamics of the air pressure variation inside the chamber is determined by the horizontal discharge $q_i$ under the partially-immersed structure.

\section{Reformulation of the model as a transmission problem}\label{Section 3: derivation_model}
This section is devoted to the reformulation of the model that we have previously derived.
More precisely, we show that \eqref{1dshallow-ext-as}-\eqref{1dshallow-int} can be written as a transmission problem 
across the structure side-walls and we recast it as a $4\times 4$ initial boundary value problem (IBVP).

\subsection{Transmission problem across the structure side-walls}\label{sec3_trans_strucure}
The transmission problem we derive here is associated with the wave-structure interaction at the vertical side-walls of the partially-immersed object. From \eqref{constant-air-pressure}-\eqref{time-dep-air-pressure} the air pressure is independent of the spatial variable both inside and outside the chamber.
Therefore, \eqref{1dshallow-ext-as} can be written as
\begin{equation}\label{1dshallow-Epluslleft}
    \begin{cases}
    \partial_t\zeta + \partial_x q=0,\\[4pt]
    \partial_t q + \partial_x \left(\dfrac{q^2}{h_0+\zeta} \right)+g (h_0+\zeta)\partial_x \zeta=0,\\[4pt]
    \end{cases}\quad \mbox{ in } \quad (0,T)\times\mathcal{E} ,
\end{equation}
 with transmission condition 
\begin{equation*}
    \qquad q_{|_{x=-r}}=q_{|_{x=r}},
\end{equation*}
and boundary conditions 
\begin{equation}\label{boundary-cond}
    \zeta_{|_{x=-l}}=\zeta_{\rm ent}, \qquad q_{|_{x=l}}=0.
\end{equation}
Moreover,  in the interior domain one has
\begin{equation}\label{pb-int}
\dfrac{d q_i}{dt}=-\dfrac{h_w}{\rho}\partial_x \underline{P} \quad \mbox{ in } \quad (0,T)\times\mathcal{I}.
\end{equation}
\begin{remark}
In \eqref{pb-int} we have implicitly used the fact that the bottom of the partially-immersed structure is flat, yielding that $\zeta_w$ is constant in space as well. More generally, 
for a solid with non-flat bottom parametrization $\zeta_w(x)$ the evolution equation for $q_i$ would read
\begin{equation*}
     \dfrac{d q_i}{dt} - q_i^2 \frac{\partial_x\zeta_w}{h^2_w}+ g h_w \partial_x \zeta_w=-\dfrac{h_w}{\rho}\partial_x \underline{P},
\end{equation*}
with $h_w(x)=h_0+\zeta_w(x).$
\end{remark}

We will see later that, after making a change of variables, the $2\times 2$ transmission problem \eqref{1dshallow-Epluslleft}-\eqref{boundary-cond} can be recast as a $4\times 4$ hyperbolic quasilinear initial boundary value problem (IBVP). It is known that a necessary condition to ensure the well-posedness of this type of problems is that the number of boundary conditions must be equal to the number of positive eigenvalues of the system (see
\cite[Section 1.1]{bastin-coron2016}, \cite{benzoni-serre2007}). In our case we will have two positive eigenvalues, the positive eigenvalue of 
$A(U)$
in $\mathcal{E}_+$ and the opposite of the negative eigenvalue of $A
(U)$ in $\mathcal{E}_-$. Unfortunately, the continuity of $q$ across the side-walls only gives us one transmission condition and an additional transmission condition is indispensable. This will be derived in the next subsection. 

\subsection{Derivation of the second transmission condition}\label{sec3_deri_second_trans_cond}
In the case of a boat, as in \cite{IguLan21}, the partially-immersed structure has non-vertical lateral walls and the second transmission is determined by the continuity of the surface elevation at the contact points where the waves, the air and the solid meet. Contrarily, in the presence of vertical side-walls, which is the case considered in this paper, the continuity of the surface elevation ceases to hold. 
However, from \eqref{pb-int} we know that the horizontal discharge  $q$ in the interior domain is equal to $q_i$ that depends only on time. Therefore the second transmission condition reads 
$q_{|_{x=\pm r}}=q_i$ or
equivalently 
\begin{equation*}
    \label{average-transcond}
    \langle q\rangle =q_i \quad \mbox{with} \quad \langle q\rangle:=\tfrac{1}{2}(q_{|_{x=-r}}+q_{|_{x=r}}).
\end{equation*} 
When the air pressure is assumed to be constant both outside and inside the chamber, the fluid-structure system can be assumed to be isolated, yielding that the total fluid-structure energy is a conserved quantity. Then, using local conservation of energy derived from the equations, one obtains an evolution equation on $q_i$ depending on the traces of the $\zeta$ and $q$ at both side walls. This has been done in \cite{bocchihevergara2021} for the nonlinear shallow water equations and in \cite{bresch2019waves} for the Boussinesq system. 
Following the same approach, we want to derive an evolution equation for $q_i$ that completely determines it and permits to close the system. Let us recall for the sake of clarity the fluid equations we are studying:
\begin{equation}\label{exteqs-as1}
    \begin{cases}
    \partial_t\zeta + \partial_x q=0,\\[5pt]
    \partial_t q + \partial_x \left(\dfrac{q^2}{h_0+\zeta} \right)+g (h_0+\zeta)\partial_x \zeta=-\dfrac{h_0+\zeta}{\rho}\partial_x P_\mathrm{air},\\[5pt]
    \end{cases}\quad \mbox{ in } \quad (0,T)\times \mathcal{E}.
\end{equation}
Notice that the source term in the second equation of \eqref{exteqs-as1} vanishes since $P_{\rm air}$ does not depend on the spatial variable, but for our analysis it is crucial to keep that term explicit.
Multiplying the first equation in \eqref{exteqs-as1} by $\rho g \zeta$ and the second equation by $\rho \frac{q}{h_0 + \zeta}$ we obtain 
\begin{equation*}
\partial_t \mathfrak{e}_{\rm ext} + \partial_x \mathfrak{f}_{\rm ext}
= P_{\mathrm{air}}   \partial_x q \qquad \mbox{ in } \quad (0,T)\times \mathcal{E},
\end{equation*}
where $\mathfrak{e}_{\rm ext}$ and $\mathfrak{f}_{\rm ext}$ are the local fluid energy and the local flux in the exterior domain respectively defined by 
\begin{equation*}
	\mathfrak{e}_{\rm ext}= \rho\frac{q^2}{2h} + g\rho \frac{\zeta^2}{2} \quad \mbox{ and }\quad  \mathfrak{f}_{\rm ext}=q\left(\rho\frac{q^2}{2h^2} + g\rho\zeta  +
P_{\text{air}}
	\right).
\end{equation*}
Next, in the interior domain the equations read
\begin{equation}\label{eqs-int1}
 \dfrac{d q_i}{dt}=-\dfrac{h_w}{\rho}\partial_x \underline{P} \quad \mbox{ in } \quad (0,T)\times\mathcal{I}.
\end{equation}
Multiplying the equation above by $\rho \dfrac{q_i}{h_w}$, we obtain the local conservation of energy in the interior domain 
\begin{equation*}\label{intcon1}
\partial_t \mathfrak{e}_{\rm int} + \partial_x \mathfrak{f}_{\rm int}=0,
\end{equation*}
where $\mathfrak{e}_{\rm int}$ and $\mathfrak{f}_{\rm int}$ are the local fluid energy and the local flux in the interior domain respectively defined by 
\begin{equation*}
\mathfrak{e}_{\rm int}=  \rho \frac{q_i^2}{2h_w} + \rho g \frac{\zeta_w^2}{2} \quad \mbox{ and }\quad  \mathfrak{f}_{\rm int}=q_i \underline{P} . 
\end{equation*}
Notice that we have used the fact that $\partial_t \zeta_w=0$ since the structure is fixed. 
Let us define the global fluid energy by $$E_{\mathrm{fluid}}= \int_{\mathcal{I}} \mathfrak{e}_{\rm int} + 
	\int_{\mathcal{E}} \mathfrak{e}_{\rm ext}.$$
Therefore, denoting the jump $\llbracket f \rrbracket := f_{|_{x=r^+}}- f_{|_{x=(-r)^-}},$ we compute that
	\begin{equation}\label{der-fsenergy}
	\begin{aligned}
	 \frac{d}{dt}{E}_{\mathrm{fluid}} 
	& = \int_{\mathcal{I}} \partial_t \mathfrak{e}_{\rm int} + 
	\int_{\mathcal{E}} \partial_t\mathfrak{e}_{\rm ext}
	\\[5pt]&= - \llbracket \mathfrak{f}_{\rm int} \rrbracket  + \llbracket
	\mathfrak{f}_{\rm ext} \rrbracket - \left(\mathfrak{f}_{\rm ext} \right)_{|_{x=l}} + 
	\left(\mathfrak{f}_{\rm ext}\right)_{|_{x=-l}} + (P_{\rm air}q)_{|_{x=l}} -(P_{\rm air}q)_{|_{x=-l}} - 
	\llbracket P_{\mathrm{air}} q \rrbracket
	\\[5pt]&= - \llbracket \mathfrak{f}_{\rm int} \rrbracket  + \llbracket \mathfrak{f}_{\rm ext} \rrbracket + \rho\left(q\big(\tfrac{q^2}{2h^2} + g\zeta\big)\right)_{|_{x=-l}} - P_{\mathrm{ch}} q_i,
	\end{aligned}
	\end{equation}
	where in the second equality we have used that $P_{\rm air}$ is constant in space and in the third equality we have used the wall boundary condition $q_{|_{x=l}}=0$, the fact that $q_{|_{x=\pm r}}=q_i$ and $\llbracket P_{\mathrm{air}} \rrbracket = P_{\rm ch} $ by definition of $P_{\rm air}$ in  $\mathcal{E}_-$ and in $   
    \mathcal{E}_+$.
    Notice that in the right-hand side of the equation above there is a term involving the air pressure variation $P_\mathrm{ch}$ inside the chamber of the OWC, whose information cannot be obtained from the fluid equations but is determined by \eqref{Pchamber-dynamics}. One can see that this ODE has an intrinsic energy $\frac{1}{2} P^2_{\rm ch}$. The second term in the left hand side of \eqref{Pchamber-dynamics} can be interpreted as a damping. Our goal is to derive a transmission condition for the transmission problem by imposing the conservation of a certain characteristic energy of the fluid-OWC coupled problem. This way of coupling physical subsystems using a power conserving interconnection can be also thought as the formulation of port-Hamiltonian systems. We refer the interested reader to \cite{Rashad-Califano-vanderSchaft_2020} for the general formulation and to \cite{vanderschaft-maschke_2002} for its approach to PDEs.\\
    Let us consider the case when no dissipation occurs in the OWC chamber and the damping term in \eqref{Pchamber-dynamics} is negligible. In this non-damped scenario, it is reasonable to ask for conservation of the total 
    fluid-OWC
    energy. We then consider the non-damped version of \eqref{Pchamber-dynamics}, namely,
     \begin{equation}\label{Pchamber-dynamics-nodamp}
     \frac{d P_\mathrm{ch}}{dt} = \gamma_2 q_i,
\end{equation}and multiplying by $P_{\rm ch }$ yields 
\begin{equation*}
     \frac{1}{2\gamma_2}\frac{d P^2_\mathrm{ch}}{dt} = P_{\rm ch} q_i. \vspace{0.5em}
\end{equation*}
Injecting the previous equality into \eqref{der-fsenergy} and defining the OWC energy\footnote[1]{Using the definition of the physical parameter $\gamma_2$, it is easy to check that the introduced quantity $E_{\rm OWC}$ is indeed homogeneous to an energy.
} $E_{\rm OWC}$ by
$$E_{\rm OWC}=\frac{1}{2\gamma_2}P^2_{\rm ch},$$
we obtain 
\begin{equation*}
   \frac{d}{dt}({E}_{\mathrm{fluid}}+ E_{\rm OWC}) =- \llbracket \mathfrak{f}_{\rm int} \rrbracket  + \llbracket \mathfrak{f}_{\rm ext} \rrbracket + \rho\left(q\big(\tfrac{q^2}{2h^2} + g\zeta\big)\right)_{|_{x=-l}}.
\end{equation*}
Now we impose that total energy 
${E}_{\mathrm{fluid}}
+ E_{\rm OWC}$
is a conserved quantity of the problem, which is defined in a bounded domain. Hence we assume that \begin{equation}\label{conserv-assum}
   \frac{d}{dt}({E}_{\mathrm{fluid}}+ E_{\rm OWC}) = \rho\left(q\big(\tfrac{q^2}{2h^2} + g\zeta\big)\right)_{|_{x=-l}}.
\end{equation}
 This is an adaptation of the conservation of total fluid-OWC energy to a bounded domain case, where the term in the right-hand side is the fluid flux at the entrance of the domain (equal to the one in \cite{bocchihevergara2021}). The wall boundary condition makes the fluid flux vanish at the end of the domain.\\
 With this assumption, we get 
 \begin{equation*}
	\llbracket \mathfrak{f}_{\rm int} \rrbracket 
	=
	\llbracket \mathfrak{f}_{\rm ext} \rrbracket.
\end{equation*}
By definition of the fluxes it follows
\begin{equation*}
\llbracket q_i \underline{P}   \rrbracket = \left\llbracket q(\rho\frac{q^2}{2h^2} + g\rho\zeta  +  P_{\mathrm{air}} ) \right\rrbracket .
\end{equation*}
Then, using again that $q_{|_{\pm r}}=q_i$, $\llbracket P_{\mathrm{air}}\rrbracket=P_{\rm ch}$, we derive from \eqref{eqs-int1} the following ODE for $q_i$:
\begin{equation}
\label{ODEqi}
   -\alpha\frac{d q_i}{dt} = \left\llbracket g \zeta + \frac{q^2}{2(h_0+\zeta)^2}\right\rrbracket +\frac{P_\mathrm{ch}}{\rho}
\end{equation}
with $\alpha=\frac{2r}{h_w}$, where $2r=|\mathcal{I}|$ is the width of the partially-immersed structure.
\begin{remark}
As previously explained, our goal is to derive a transmission condition that allows to close the system in the case of a partially-immersed structure with vertical side-walls. The ODE for $q_i$ was derived by considering the non-damped version \eqref{Pchamber-dynamics-nodamp} of the original ODE \eqref{Pchamber-dynamics} and by assuming the existence of a reasonable conserved quantity in that particular case. However, the derivation of the transmission condition is independent of the effective conservation of the total fluid-OWC energy in the real scenario, where damping occurs. Indeed, after having obtained the condition $\llbracket \mathfrak{f}_{\rm int} \rrbracket =	\llbracket \mathfrak{f}_{\rm ext} \rrbracket$, one should consider the original ODE \eqref{Pchamber-dynamics}. Then, instead of \eqref{conserv-assum}, it would yield
\begin{equation*}
   \frac{d}{dt}({E}_{\mathrm{fluid}}+ E_{\rm OWC}) = -\frac{\gamma_1}{\gamma_2} P^2_{\rm ch} + \rho \left(q\big(\tfrac{q^2}{2h^2} + g\zeta\big)\right)_{|_{x=-l}},
\end{equation*}
which shows dissipation of the considered energy. The dissipated energy is crucial for the  good implementation of the wave energy converter as it is captured by the device and transformed via the turbine into electric energy.
\end{remark}

\begin{remark}
The ODE \eqref{ODEqi} is a generalization of the one derived in \cite{bocchihevergara2021} by the authors. Indeed, considering the air pressure equal to the constant atmospheric pressure also inside the chamber, one has $P_\mathrm{ch}\equiv 0$ and the same equation as in \cite{bocchihevergara2021} is recovered.
\end{remark}
\noindent Then the transmission problem  \eqref{1dshallow-Epluslleft}-\eqref{boundary-cond} reads
				\begin{equation}\label{transproblemcomplete}
    \begin{cases}
    \partial_t\zeta + \partial_x q=0 \\[4pt]
    \partial_t q + \partial_x \left(\dfrac{q^2}{h_0+\zeta} \right)+g (h_0+\zeta)\partial_x \zeta=0
    \end{cases} \mbox{ in } \quad (0,T)\times\mathcal{E},\\[4pt]
\end{equation}
with  boundary conditions
\begin{equation*}
    \zeta_{|_{x=-l}}=\zeta_{\rm ent}(t), \qquad q_{|_{x=l}}=0
\end{equation*}
and transmission conditions	
\begin{equation*}
\llbracket q \rrbracket=0, \quad \langle q\rangle=q_i,
\end{equation*}	
where $q_i$, $P_\mathrm{ch}$ satisfy
\begin{equation}\label{equation_P_q}
\begin{cases}
\begin{aligned}
 &\frac{d q_i}{dt} = -\frac{1}{\alpha}\left\llbracket g \zeta + \frac{q^2}{2(h_0+\zeta)^2}\right\rrbracket -\frac{P_\mathrm{ch}}{\alpha\rho},
 \\[4pt]
    & \frac{d P_\mathrm{ch}}{dt} =- \gamma_1 P_\mathrm{ch}+ \gamma_2 q_i.
          \end{aligned}
     \end{cases}
\end{equation}
The initial conditions of the problem are  
\begin{equation}\label{initial-cond}
 \zeta(0,x)=\zeta_0(x), \quad q(0,x)=q_0(x)  \quad \mbox{in}\quad  \mathcal{E}, \qquad \text{ and }\qquad 
q_i(0)=q_{i,0}, \quad P_\mathrm{ch}(0)=P_{\mathrm{ch},0}. 
\end{equation}

\subsection{Reduction of the
transmission problem across the structure to an IBVP}\label{sec3_transmiss_pron_as_IBVP}
In this subsection we show how the $2\times 2$ transmission problem \eqref{transproblemcomplete}-\eqref{initial-cond} can be reduced to a $4\times 4$ one-dimensional quasilinear IBVP with a semilinear boundary condition. 
First, we rewrite \eqref{transproblemcomplete}-\eqref{initial-cond} in the compact form
\begin{equation}
\begin{aligned}
\begin{cases}
\partial_t U + 
A
(U)\partial_x U = 0 &\quad\mbox{in}\quad  (0,T)\times \mathcal{E},\\[5pt]
U(0,x)=U_0(x) &\quad\mbox{in}\quad \mathcal{E},\\[5pt]
M^+U_{
|_{x=r}
}- M^-U_{
|_{x=-r}
}=V(G(t)) &\quad\mbox{in}\quad  (0,T),\\[5pt]
e_1\cdot U_{|_{x=-l}}=g^{(1)}(t), \quad  e_2\cdot U_{|_{x=l}}=g^{(2)}(t) &\quad\mbox{in}\quad  (0,T),
\end{cases}
\end{aligned}
\label{transpb2x2}
\end{equation}
with $U(t,x)=(\zeta(t,x),q(t,x))^T$, the matrices
$$ 
A(U)=\left(\begin{matrix}
			0& 1 \\[5pt]
			g(h_0+\zeta) -\dfrac{q^2}{(h_0+\zeta)^2}& \dfrac{2q}{h_0+\zeta}
				\end{matrix}
				\right), \qquad M^\pm=\left(
\begin{array}{rr}
0& 1 \\
0& \pm\frac{1}{2}
\end{array}
\right),$$
the boundary data $g(t)= (g^{(1)}(t), g^{(2)}(t))=(\zeta_{\rm ent}(t),0)^T$ and 
$G(t)=(q_i(t),P_\mathrm{ch}(t))^T$ that satisfies the evolution equation
\begin{equation}
\begin{cases}
\dot{G}= \Theta \left(G, U_{|_{x=\pm r}
}\right),\\[5pt]
G(0)=G_0.
\end{cases}
\label{semilinearODE}
\end{equation}
The initial data are  $$U_0(x)=(\zeta_0(x),q_0(x))^T, \qquad G_0= (q_{i,0},P_{\rm ch,0})^T.\\[5pt]$$
For $U=(U^{(1)},U^{(2)})^T$, $G=(G^{(1)},G^{(2)})^T$ and  $\Theta=(\Theta^{(1)}, \Theta^{(2)})^T,$ we have  $V(G)=(0,G^{(1)})^T$ and 
\begin{equation*}
    \begin{aligned}
    \Theta^{(1)} (G, U_{|_{x=\pm r}})
    =-\frac{1}{\alpha}\left[( g U^{(1)} + \frac{(U^{(2)})^2}{2(h_0+U^{(1)})^2})_{
    \big|_{x=r}
    } -( g U^{(1)} + \frac{(U^{(2)})^2}{2(h_0+U^{(1)})^2})_{
    \big|_{x=-r}
    }\right] -\frac{G^{(2)}}{\alpha\rho},
     \end{aligned}
    \end{equation*}
    and 
\begin{equation*}
     \Theta^{(2)} \left(G, U_{|_{x=\pm r}}
     \right)  =- \gamma_1 G^{(2)}+ \gamma_2 G^{(1)}.
\end{equation*}
Equation \eqref{semilinearODE} has the same form of the kinematic-type evolution equation considered in \cite{IguLan21} where the authors dealt with a free boundary transmission problem. Here, although we consider a fixed boundary transmission problem, the same situation occurs: the derivative of $G$ has the same regularity as the trace of the solution at the boundary. 
The boundary condition is semilinear, in the sense that the evolution equation \eqref{semilinearODE} is nonlinear only on the trace of the solution at the boundary and not on its derivatives. This would be the case when considering a boat-type structure, which turns out to be a free boundary hyperbolic problem. A kinematic-type evolution equation for the moving contact points $x_\pm(t)$ can be derived after time-differentiating the boundary condition $U(t,x_\pm (t))=U_i(t,x_\pm (t))$, where $U_i$ is a known function. 
In the nonlinear equation obtained, there are terms involving traces of derivatives $\partial U_{
|_{x=\pm r}
}$ and the boundary condition is fully nonlinear because there is a loss of one derivative in the estimates (see\cite{IguLan21}). Here we deal with a less singular evolution equation.

Let us now recast \eqref{transpb2x2}-\eqref{semilinearODE} as an IBVP by introducing a change of variable $x^\prime=-x$ on the spatial space 
$(-l,-r)$
and writing 
\begin{align*}
&u^+(t,x)=U(t,x), \quad u^-(t,x)=U(t,-x),\\& u_0^+(x)=U_0(x), \qquad u_0^-(x)=U_0(-x).
\end{align*}
Thus, the system \eqref{transpb2x2} is equivalent to the following $4\times4$ quasilinear hyperbolic system in $\Omega_T:=(0,T)\times \mathcal{E}_+$, where $\mathcal{E}_+=(r,l)$,
\begin{equation}
		\begin{aligned}
		\begin{cases}
		\partial_t u + \mathcal{A}(u)\partial_x u = 0 &\quad\mbox{in}\quad  \Omega_T,\\[5pt]
		u(0)=u_0(x) &\quad\mbox{in}\quad  \mathcal{E}_+,\\[5pt]
		\mathcal{M}_r u_{|_{x=r}}=V(G(t)) &\quad\mbox{in}\quad  (0,T),
		\\[5pt]
		\mathcal{M}_l u_{|_{x=l}}=g(t) &\quad\mbox{in}\quad  (0,T),
		\end{cases}
		\end{aligned}
		\label{IBVP4x4}
\end{equation}
		
\noindent where $u=(u^-,u^+)^T$, $u_0=(u_0^-,u_0^+)^T$ are $\R^4$-valued functions and\\
$$\mathcal{A}(u)=\mathrm{diag}\left(-A(u^-) ,  A(u^+)\right), \qquad
\mathcal{M}_{
r
}
=\left(-M^- \ M^+\right), \qquad \mathcal{M}_l=\left(\begin{matrix}
1&0&0&0\\
0&0&0&1
\end{matrix}\right),$$\\
are respectively one $4\times4$ matrix and two $2\times 4$ matrices. Moreover, the ODE \eqref{semilinearODE} reads
\begin{equation}
		\begin{cases}
		\dot{G}= \Theta(G, u_{
		|_{x=r}
		} ),\\[5pt]
		G(0)=G_0.
		\end{cases}
\label{semilinearODE4x4}
\end{equation}
In the next section we will study this  IBVP with semilinear boundary condition in a general setting and we will investigate its local well-posedness. 
		
\section{1d Kreiss-symmetrizable hyperbolic IBVPs}\label{KreissHypIBVP}
In this section we study general one-dimensional quasilinear hyperbolic IBVP with a semilinear boundary condition as \eqref{IBVP4x4}-\eqref{semilinearODE4x4}. 
In general, one-dimensional hyperbolic initial boundary value problems are treated by using the method of characteristics and local well-posedness in $C^1$
(see \cite{li-yu1985}, and references therein).
In the Sobolev setting, multi-dimensional results are employed at the cost of high regularity requirements for initial data and derivatives loss with respect to the boundary and initial data. These drawbacks were recently removed in \cite{bocchi2020floating,IguLan21} by taking advantage of the specificities of the one-dimensional case.
Following the argument in \cite{bocchi2020floating,IguLan21} we establish local-in-time well-posedness for Kreiss symmetrizable systems, that is Friedrichs symmetrizable systems whose symmetrizer yields maximal dissipativity on the boundary. This property permits us to 
gain one derivative on the control of the trace of the solution at the boundary and it will be crucial to close the energy estimates needed to apply an iterative scheme argument to get a local well-posedness result.\\
In order to study quasilinear hyperbolic IBVP with a boundary data determined by an evolution equation, we need first to consider linear hyperbolic IBVP with a given boundary data. We will then use the estimates derived from the linear theory for the ``PDE part" and nonlinear estimates for the ``ODE part" to show that the sequence of approximated solutions defined by the iterative scheme is bounded and convergent in some proper spaces. The limit of the sequence will be then the unique solution of the quasilinear problem.

\subsection{Variable-coefficients linear hyperbolic IBVPs}\label{subsection_linearsystem}
In this subsection we deal with linear hyperbolic IBVP with variable coefficients. Let us present some linear energy estimates together with a well-posedness result for a Kreiss-symmetrizable system,
whose definition will be given in the sequel. To do this, we consider the following linear
hyperbolic initial boundary value problem 
\begin{equation}\label{linear_IBVP}
\begin{aligned}
\begin{cases}
\partial_t u + \mathcal{A}(\widetilde{u})\partial_x u = f &\quad\mbox{in}\quad  \Omega_T,\\[5pt]
u(0)=u_0(x) &\quad\mbox{in}\quad  \mathcal{E}_+,\\[5pt]
\mathcal{M}_r u_{
|_{x=r}
}=V(t) &\quad\mbox{in}\quad  (0,T), 
	\\[5pt]
		\mathcal{M}_l u_{|_{x=l}}=g(t) &\quad\mbox{in}\quad  (0,T),
\end{cases}
\end{aligned}
\end{equation}
where $u=u(t,x)$, $u_0$, $\widetilde{u}= \widetilde{u}(t,x)$ and $f=f(t,x)$ are given $\mathbb{R}^4$-valued functions, 
 $\mathcal{A}(\widetilde{u})\in \mathcal{M}_4(\mathbb{R})$, $\mathcal{M}_r$, $\mathcal{M}_l\in \mathcal{M}_{2,4}(\mathbb{R})$ are given constant matrices, 
$V$ and $g$ are given $\R^2$-valued functions.
Let us introduce the definition of Kreiss symmetrizer for a system. 
\begin{definition}\label{Kreiss-symmetrizable}
The hyperbolic initial boundary value problem \eqref{linear_IBVP} is \textit{Kreiss symmetrizable} if there exists a symmetric matrix $\mathcal{S}(x, \widetilde{u})\in \mathcal{M}_4(\R)$, called \textit{Kreiss symmetrizer}, such that $\mathcal{S}(x, \widetilde{u})\mathcal{A}(\widetilde{u})$ is symmetric and the following properties hold:
\begin{enumerate}
\item \label{def-dis}There exist constants $c_1,C_1>0$ such that
$$c_1 |v|^2\leq v^T \mathcal{S}(x, \widetilde{u}) v \leq C_1 |v|^2$$
for any $v\in\R^4$ and $x\in\mathcal{E}_+$.
\item There exist constants $c_2,c_3,C_2,C_3>0$ such that the boundary conditions are \textit{maximal dissipative}, \textit{i.e.}
\begin{align*}v^T (\mathcal{S}(r, \widetilde{u}_{|_{x=r}})\mathcal{A}(\widetilde{u}_{|_{x=r}}))v &\leq -c_2|v|^2+C_2 |\mathcal{M}_{
r
}
v|^2,\\[5pt]
-v^T (\mathcal{S}(l, \widetilde{u}_{|_{x=l}})\mathcal{A}(\widetilde{u}_{|_{x=l}}))v &\leq -c_3|v|^2+C_3 |\mathcal{M}_l v|^2,\end{align*}
for any $v\in\R^4$.
\item \label{bound} There exists a constant $\beta>0$ such that
$$\|\partial_t \mathcal{S}+ \partial_x (\mathcal{S}\mathcal{A})\|_{L^\infty(\Omega_T)}\leq \beta.$$
\end{enumerate}
\end{definition}

A Kreiss symmetrizer is therefore a Friderichs symmetrizer, used in the Cauchy problem theory for hyperbolic systems (see  \cite{benzoni-serre2007}), yielding maximal dissipativity on the boundaries. In order to construct this symmetrizer, we need the following assumptions on the matrix $\mathcal{A}$, the boundary matrices $\mathcal{M}_{
r
}$ and $\mathcal{M}_l$ and matrices called \textit{Lopatinski\u{i}} matrices (see \cite{IguLan21}).

\begin{assumption}\label{assumWPlin}
		Let $\widetilde{u}=(\widetilde{u}^-, \tilde{u}^+)^T$ 
  take values in $ \mathcal{U}=\mathcal{U}^- \times \mathcal{U}^+$ with $\mathcal{U}^-$, $\mathcal{U}^+$ two open sets in $\R^2$.
  There exists a constant $\kappa_0>0$ such that the following properties are satisfied: 
\begin{enumerate}
				\item \label{rankM}$\mathcal{A} \in C^\infty(\mathcal{U})$, $\det(\mathcal{M}_{
				r
				} \mathcal{M}_{
				r
				}^T)\geq \kappa_0$ and $\det(\mathcal{M}_l \mathcal{M}_l^T)\geq \kappa_0.$\\[0.1em]
				\item  $\mathcal{A}(\widetilde{u})=\mbox{diag}(-A^-(\widetilde{u}^-), A^+(\widetilde{u}^+))$ where $A^-(\widetilde{u}^-)$ and $A^+(\widetilde{u}^+)$ have eigenvalues $\pm\lambda_\pm(\widetilde{u}^-)$ and $\pm\lambda_\pm (\widetilde{u}^+)$ respectively. Furthermore, $\tilde{u}$ takes values in a compact and convex set
				 $\mathcal{K}_0\subsetneq \mathcal{U}$ and 
					$$\lambda_\pm(\widetilde{u}^-)\geq \kappa_0, \  \ \lambda_\pm(\widetilde{u}^+)\geq \kappa_0.
				 $$ 
				\item \label{UKL} Let us define the $2\times 2$ 
				 Lopatinski\u{i} matrices $\mathcal{L}_{
				r
				}
	 	(\widetilde{u}_{|_{x=r}})$ and $\mathcal{L}_l(\widetilde{u}_{|_{x=l}})$ respectively by
				\begin{align*}
				\mathcal{L}_{
				r
				}
			(\widetilde{u}_{|_{x=r}})&=\mathcal{M}_{
				r
				} 
			E(\widetilde{u}_{|_{x=r}})\quad \mbox{with}\quad  E(\widetilde{u}_{|_{x=r}})=	\left(
				\begin{matrix}
				e_-( 
    { \widetilde{u}^-}\,_{
				|_{x=r}
				})& 0_{2\times 1} \\
				0_{2 \times 1} & e_+(
    { \widetilde{u}^+}\,_{
			 |_{x=r}})
				\end{matrix}
				\right),\\
				\mathcal{L}_l(\widetilde{u}_{|_{x=l}})&=\mathcal{M}_l 
				E(\widetilde{u}_{|_{x=l}})\quad \mbox{with}\quad  E(\widetilde{u}_{|_{x=l}})=	\left(
				\begin{matrix}
				e_+(
   { \widetilde{u}^-}\,_{|_{x=l}}
    )& 0_{2\times 1} \\
				0_{2 \times 1} & e_-(
    { \widetilde{u}^+}\,_{|_{x=l}}
    )
				\end{matrix}
				\right),
				\end{align*}
				where $e_\pm(
    {\widetilde{u}^-}\,_{|_{x=r,l}})$ are the unit eigenvectors of $A^-(
    {\widetilde{u}^-}\,_{|_{x=r,l}})$ associated with the eigenvalues $\pm\lambda_\pm(
    {\widetilde{u}^-}\,_{|_{x=r,l}})$ and $ e_\pm(
{\widetilde{u}^+}\,_{|_{x=r,l}})$ are the unit eigenvectors of $A^+ ({\widetilde{u}^+}\,_{|_{x=r,l}})$ associated with the eigenvalues $\pm\lambda_\pm(
{\widetilde{u}^+}\,_{|_{x=r,l}})$. 
				Then, 	 $\mathcal{L}_{
				r
	}(\widetilde{u}_{|_{x=r}})$ and $\mathcal{L}_l(\widetilde{u}_{|_{x=l}})$ are invertible and
			$$\|\mathcal{L}_{
			r
			}
			(\widetilde{u}_{|_{x=r}})^{-1}\|_{\mathbb{R}^2 \to \mathbb{R}^2}\leq \frac{1}{\kappa_0}, \qquad \|\mathcal{L}_l(\widetilde{u}_{|_{x=l}})^{-1}\|_{\mathbb{R}^2 \to \mathbb{R}^2}\leq \frac{1}{\kappa_0}.$$ 
\end{enumerate}
\end{assumption}
Notice that the positivity of the determinants in condition \eqref{rankM} means that the rank of both matrices is $2$. This means that we have exactly two boundary conditions at $x=r$ and two boundary conditions at $x=l$. The condition \eqref{UKL} of Assumption \ref{assumWPlin} is a reformulation of the \textit{uniform Kreiss-Lopatinski\u{i}} condition, which can be derived as a stability condition on the normal mode solutions for the  \eqref{linear_IBVP} with fixed coefficients. We refer to \cite{benzoni-serre2007} for more details on this type of condition. In the general multi-dimensional theory the construction of a Kreiss symmetrizer from the uniform Kreiss-Lopatinski\u{i} condition is delicate and involved. Indeed, it requires refined paradifferential calculus and the symmetrizer obtained is a matrix-valued function depending homogeneously on space-time frequencies, hence a symbol. 
Instead, the problem that we are considering in this article is one-dimensional and we can take advantage of the specificities of the one-dimensional
setting to construct a Kreiss symmetrizer in the sense of Definition \ref{Kreiss-symmetrizable}. In the case of a problem on the half-line, the argument of \cite{bocchi2020floating,IguLan21} gives explicitly the symmetrizer. 
However, one cannot apply directly the theory in the half-line case to the bounded interval case, and the latter is not trivial. We therefore develop an adapted theory below.
\begin{lemma}\label{existenceKS}
   Assume that Assumption \ref{assumWPlin} holds for some $\kappa_0>0$. Then, there exist a matrix $\mathcal{S}(x, \widetilde{u})
   $ and positive constants $c_1, C_1,  c_2,C_2,\beta$ such that \eqref{def-dis}-\eqref{bound} in Definition \ref{Kreiss-symmetrizable} are satisfied.
\end{lemma}
\begin{proof}
From property (2) of Assumption \ref{assumWPlin}, we know that $\mathcal{A}(\widetilde{u})$ is diagonalizable. We denote its positive eigenvalues by $ \lambda_{+,j}(\widetilde{u})$ and its negative eigenvalues by $ -\lambda_{-,j}(\widetilde{u})$ for $j=1,2$. Then, $ \Pi_{\pm,j}(\widetilde{u})$ are the eigenprojectors associated with the eigenvalues $\pm \lambda_{\pm,j}(\widetilde{u})$. We construct the symmetrizer as
\begin{equation*} 
\mathcal{S}( x,\widetilde{u}) :=  W_+(x)\sum\limits_{j=1}^2 
\Pi_{+,j}(\widetilde{u})
\Pi_{+,j}(\widetilde{u})^{ T} +W_-(x)\sum\limits_{j=1}^2 \Pi_{-,j}(\widetilde{u})^{ T}\Pi_{-,j}(\widetilde{u}) ,
\end{equation*}
where $W_\pm$ are some positive smooth functions such that 
$$W_-(
r
) \gg W_+ (r) \quad \mbox{and}\quad W_+(l) \gg W_-(l).$$
Using the same decomposition of $\mathcal{A}$ as in \cite{IguLan21}, we have
\begin{align*}
& \mathcal{S}(x,\widetilde{u})\mathcal{A}(\widetilde{u})\\&=W_+(x)\sum\limits_{j=1}^2 \lambda_{+,j}(\widetilde{u}) \Pi_{+,j}(\widetilde{u})^{ T}\Pi_{+,j}(\widetilde{u}) -W_-(x)\sum\limits_{j=1}^2 \lambda_{-,j}(\widetilde{u})\Pi_{-,j}(\widetilde{u})^{ T} \Pi_{-,j}(\widetilde{u}). \end{align*}
We start by proving that, for $v\in \Ker \mathcal{M}_{
r
}$,  
\begin{equation}\label{maxdisker0}
    |v|^2\leq - C v^T\mathcal{S}( r,\widetilde{u}_{|_{x=r}}) \mathcal{A}(\widetilde{u}_{|_{x=r}})v,
\end{equation}
and for $v\in \Ker \mathcal{M}_l$,
\begin{equation}\label{maxdisker1}
    |v|^2\leq  C v^T \mathcal{S}( l, \widetilde{u}_{|_{x=l}}) \mathcal{A}(\widetilde{u}_{|_{x=l}}) v.
\end{equation}
Let us decompose $v$ as 
$$v= \sum\limits_{j=1}^2 \Pi_{-,j}(\widetilde{u})v + \sum\limits_{j=1}^2 \Pi_{+,j}(\widetilde{u})v.
$$
On the one hand, we compute that
\begin{align*}&-v^T  \mathcal{S}(r,\widetilde{u}_{|_{x=r}})\mathcal{A}(\widetilde{u}_{|_{x=r}})v\\&=-W_+(
r
) \sum\limits_{j=1}^2 \lambda_{+,j}(\widetilde{u}_{|_{x=r}})|\Pi_{+,j} (\widetilde{u}_{|_{x=r}})v|^2  + W_-(
r
)\sum\limits_{j=1}^2 \lambda_{-,j}(\widetilde{u}_{|_{x=r}})|\Pi_{-,j} (\widetilde{u}_{|_{x=r}}) v|^2.\end{align*}
For $v\in \Ker \mathcal{M}_{r}$ and using the invertibility assumption of the Lopatinski\u{i} matrix $\mathcal{L}_{r} $, we know from \cite{IguLan21} that \begin{equation}\label{control}\sum\limits_{j=1}^2 |\Pi_{+,j}(\widetilde{u}_{|_{x=r}})v|^2 \leq C \sum\limits_{j=1}^2 |\Pi_{-,j}(\widetilde{u}_{|_{x=r}})v|^2, 
\end{equation}
for some constant $C$ depending on $\|\mathcal{M}_{
r
}\|_{\mathbb{R}^4\rightarrow \mathbb{R}^2}$ and $\tfrac{1}{\kappa_0}.$ Using the uniform lower bound of $\lambda_{-,j}$ it follows
\begin{align*}
&W_-(
r
)\kappa_0\sum\limits_{j=1}^2 |\Pi_{-,j} (\widetilde{u}_{|_{x=r}})v|^2 \\&\leq -v^T \mathcal{S}(r,\widetilde{u}_{|_{x=r}})\mathcal{A}(\widetilde{u}_{|_{x=r}})v +W_+(
r
) \max_{j\in\{1,2\}} \sup_{t\in(0,T)}(\lambda_{+,j}(\widetilde{u}_{|_{x=r}}))\sum\limits_{j=1}^2 |\Pi_{+,j} ( \widetilde{u}_{|_{x=r}})v|^2 \\[1pt]
&\leq -v^T \mathcal{S}(r,\widetilde{u}_{|_{x=r}}) \mathcal{A}(\widetilde{u}_{|_{x=r}})v +C W_+(
r
) \max_{j\in\{1,2\}}\sup_{t\in(0,T)}(\lambda_{+,j}(\widetilde{u}_{|_{x=r}})) \sum\limits_{j=1}^2|\Pi_{-,j} (\widetilde{u}_{|_{x=r}}) v|^2,
\end{align*}
where in the second inequality we have used \eqref{control}. Then, since $W_-(
r
)$ is sufficiently larger than $W_+(
r
)$, there exists $c>0$  such that 
\begin{align}\label{diss-minus}
\sum\limits_{j=1}^2 |\Pi_{-,j}(\widetilde{u}_{|_{x=r}})v|^2 \leq - c\ v^T  \mathcal{S}( r,\widetilde{u}_{|_{x=r}})\mathcal{A}(\widetilde{u}_{|_{x=r}})v.
\end{align}
From the decomposition of $v$ and using again \eqref{control} we get
\begin{align*}
|v|^2=\sum\limits_{j=1}^2 |\Pi_{-,j}(\widetilde{u}_{|_{x=r}})v|^2  + \sum\limits_{j=1}^2 |\Pi_{+,j}( \widetilde{u}_{|_{x=r}})v|^2\leq (C+1)\sum\limits_{j=1}^2 |\Pi_{-,j}(\widetilde{u}_{|_{x=r}})v|^2 ,
\end{align*}and the desired estimate follows from \eqref{diss-minus}.
 On the other hand, we compute 
\begin{align*}
&v^T \mathcal{S}(l,\widetilde{u}_{|_{x=l}})\mathcal{A}(\widetilde{u}_{|_{x=l}})
v\\&= W_+(l) \sum\limits_{j=1}^2 \lambda_{+,j}( \widetilde{u}_{|_{x=l}})|\Pi_{+,j}(\widetilde{u}_{|_{x=l}})v|^2  - W_-(l)\sum\limits_{j=1}^2 \lambda_{-,j}(\widetilde{u}_{|_{x=l}})|\Pi_{-,j}(\widetilde{u}_{|_{x=l}})v|^2.\end{align*}
For $v\in \Ker \mathcal{M}_l$ and using the invertibility assumption of the Lopatinski\u{i} matrix $\mathcal{L}_l $ from \cite{IguLan21} we know that \begin{equation}\label{control2}\sum\limits_{j=1}^2 |\Pi_{-,j}(\widetilde{u}_{|_{x=l}})v|^2 \leq C \sum\limits_{j=1}^2 |\Pi_{+,j}(\widetilde{u}_{|_{x=l}})v|^2, 
\end{equation} 
for some constant $C$ depending on $\|\mathcal{M}_l\|_{\mathbb{R}^4\rightarrow \mathbb{R}^2}$ and $\tfrac{1}{\kappa_0}.$ Using the uniform lower bound of $\lambda_{+,j}$, it follows
\begin{align*}&W_+(l)\kappa_0\sum\limits_{j=1}^2 |\Pi_{+,j}(\widetilde{u}_{|_{x=l}})v|^2 \\&\leq v^T \mathcal{S}( l,\widetilde{u}_{|_{x=l}})\mathcal{A}(\widetilde{u}_{|_{x=l}})v +W_-(l) \max_{j\in\{1,2\}} \sup_{t\in(0,T)}(\lambda_{-,j}(\widetilde{u}_{|_{x=l}}))\sum\limits_{j=1}^2 |\Pi_{-,j} (\widetilde{u}_{|_{x=l}}) v|^2 \\[1pt]
&\leq v^T \mathcal{S}(l,\widetilde{u}_{|_{x=l}})\mathcal{A}(\widetilde{u}_{|_{x=l}})v +C W_-(l) \max_{j\in\{1,2\}} \sup_{t\in(0,T)}(\lambda_{-,j}(\widetilde{u}_{|_{x=l}})) \sum\limits_{j=1}^2|\Pi_{+,j}(\widetilde{u}_{|_{x=l}})v|^2,
\end{align*}
where in the second inequality we have used \eqref{control2}. Then, since $W_+(l)$ is sufficiently larger than $W_-(l)$, there exists $c>0$ such that 
\begin{align}\label{diss-minus2}
\sum\limits_{j=1}^2 |\Pi_{+,j} (\widetilde{u}_{|_{x=l}})v|^2 \leq  c\ v^T \mathcal{S}(l,\widetilde{u}_{|_{x=l}})\mathcal{A}(\widetilde{u}_{|_{x=l}})v.
\end{align}
From the decomposition of $v$ and using again \eqref{control2}, we get
\begin{align*}
|v|^2=\sum\limits_{j=1}^2 |\Pi_{-,j} (\widetilde{u}_{|_{x=l}})v|^2 + \sum\limits_{j=1}^2 |\Pi_{+,j} (\widetilde{u}_{|_{x=l}})v|^2\leq (C+1)\sum\limits_{j=1}^2 |\Pi_{+,j} (\widetilde{u}_{|_{x=l}})v|^2 ,
\end{align*}and the desired estimate follows from \eqref{diss-minus2}. 
Finally, one can repeat the same argument used in \cite{IguLan21} and exploit \eqref{maxdisker0}-\eqref{maxdisker1} to obtain both estimates in property $(2)$ of Definition \ref{Kreiss-symmetrizable} for any $v\in \mathbb{R}^4$.
\end{proof}

Therefore, in the well-posedness theorem for the linear initial boundary value problem \eqref{linear_IBVP} we will only assume Assumption \ref{assumWPlin}. Before stating the result, we shall introduce the notion of compatibility conditions for the data of \eqref{linear_IBVP}.
\paragraph{\textbf{Compatibility conditions.}}
In order to have continuous solutions in time and space, the boundary data at initial time must match the boundary conditions at initial time. That is, on the edges $(t,x)=(0,
r
)$ and $(t,x)=(0,l)$ the initial data $u_0$ and boundary data $V, g$ must satisfy
\begin{equation}
\mathcal{M}_{
r
}{u_0}_{|_{x=
r
}}=V_0, \qquad \mathcal{M}_l{u_0}_{|_{x=l}}=g_0,
\label{lin_compatibility-0}
\end{equation} 
with $V_0=V(0)$ and $g_0=g(0)$.  Analogously, defining $u_1=\partial_t u(0,x)$, $V_1= \dot{V}(0)$ and $g_1=\dot{g}(0)$, $C^1$-solutions must satisfy \eqref{lin_compatibility-0} together with 
\begin{equation*}
\mathcal{M}_{r}{u_1}_{|_{x=r}}=V_1, \qquad \mathcal{M}_l{u_1}_{|_{x=l}}=g_1.
\label{lin_compatibility-1}
\end{equation*}
More generally, let us define  $u_k=\partial_t^k u(0,x)$,  
$V_k= V^{(k)}(0)$ and $g_k= g^{(k)}(0)$for $k\geq 0$. Then, smooth enough solutions must satisfy
\begin{equation}
\mathcal{M}_{r} {u_k}_{|_{x=r}}=V_k, \qquad \mathcal{M}_l {u_k}_{|_{x=l}}=g_k.
\label{lin_compatibility-k}
\end{equation}
Let us now define $f_k=\partial_t^k f(0,x).$ Using the evolution equation in \eqref{linear_IBVP} and applying an inductive argument, we can write ${u_k}$
as a function only in terms of the initial data $u_0$ 
and the source term $f$, namely
\begin{equation*}
		{u_{k}}=\mathcal{C}_{\widetilde{u}_{0,...,k-1}}( u_0, f_0, ..., f_{k-1}) \qquad \mbox{for} \quad k\geq 1,
\label{lin_uk}
\end{equation*}
where $\mathcal{C}_{\widetilde{u}_{0,...,k-1}}(u_0, f_0, ..., f_{k-1})$ is a smooth function of $\partial_x^{j+1} u_0$, $\partial_x ^{k-1-j}f_j$ for $j\!=\!0,...,k-1$ and its coefficients depend on $\widetilde{u}_0$, $\partial_x ^{k-1-j}\widetilde{u}_j$ for $j\!=\!0,...,k-1$.
The function $\mathcal{C}_{\widetilde{u}_{0,...,k-1}}(u_0, f_0, ..., f_{k-1})$ can be written in an explicit way by repeatedly using Faá di Bruno's formula. As it is not
relevant to our analysis, we only give its explicit expression for $k=1,2$ and we refer the reader to \cite{di1857note} for more details. They read
\begin{align*}
&\mathcal{C}_{\widetilde{u}_{0}}( u_0, f_0)= -\mathcal{A}(\widetilde{u}_0)\partial_x u_0 + f_0,
\\[2pt]
\mathcal{C}_{\widetilde{u}_{0,1}}(u_0, f_0, f_1)\!=\! (-D_\mathcal{A}(\widetilde{u}_0)
\!\cdot\!\widetilde{u}_1 & + \mathcal{A}(\widetilde{u}_0)D_\mathcal{A}(\widetilde{u}_0)\!\cdot \!\partial_x\widetilde{u}_0) \partial_x u_0 + \mathcal{A}^2(\widetilde{u}_0)\partial_{xx}u_0 - \mathcal{A}(\widetilde{u}_0)\partial_x f_0 + f_1.
\end{align*}
The compatibility conditions above permit us to introduce the following definition. 
\begin{definition}\label{lin_compatibility}
			Let $m\geq 1$ be an integer. The data $u_0\in H^m(
			r
			,l)$, $f\in H^m(\Omega_T)$ and $V,g\in H^m(0,T)$ of the linear initial boundary value problem \eqref{linear_IBVP} satisfy the compatibility conditions up to order $m-1$ if \eqref{lin_compatibility-k} holds for $k = 0, 1, ..., m -1.$
\end{definition}
\noindent We can now state the well-posedness theorem for the linear IBVP \eqref{linear_IBVP} with given boundary data.

\begin{theorem}\label{thm_Iguchi_lannes}
Let $m \geq 1$ be an integer and $T >0$. Assume that Assumption \ref{assumWPlin} holds for some $\kappa_0>0$ and that there exist constants $0 < K_0 \leq K$ such that
\begin{equation*}
\frac{1}{\kappa_0}, \|\mathcal{A}\|_{L^{\infty}(\mathcal{K}_0)}, \|\mathcal{M}_{
r
}\|_{\mathbb{R}^4\rightarrow \mathbb{R}^2}, \|\mathcal{M}_l\|_{\mathbb{R}^4\rightarrow \mathbb{R}^2} \leq K_0,  \qquad
\|\mathcal{A}\|_{W^{m,\infty}(\mathcal{K}_0)}, \ \|\widetilde{u}\|_{W^{1, \infty} (\Omega_T) \cap \mathbb{W}^m (T)} \leq K.
\end{equation*}
Then, for any data $u_0 \in H^m (
\mathcal{E}_+)$, $ V,g \in H^{m} (0, T)$, and $f \in H^m (\Omega_T)$ satisfying the compatibility conditions up to order $m-1$ in the sense of Definition \ref{lin_compatibility}, there exists a unique solution  $u\in \mathbb{W}^m (T)$ to the initial boundary value problem \eqref{linear_IBVP}. Moreover, the following inequality holds for any $t\in [0, T]$:
\begin{equation}\label{estimate_thm1}\begin{aligned}
  &\vertiii{u(t)}_m + |u_{|_{x=
  r
  ,l}}|_{m,t} \\
 &\leq C(K_0)e^{C(K)t} \left(\vertiii{u(0)}_m+ |(V,g)|_{H^m(0,t)} + |f_{|_{x=
 r
 ,l}}|_{m-1, t} + \int_{0}^{t} \vertiii{f(t')}_{m} dt' \right).
 \end{aligned}
\end{equation}
\end{theorem}

We will apply Theorem \ref{thm_Iguchi_lannes} later in order to prove the well-posedness result for the quasilinear IBVP \eqref{IBVP4x4}-\eqref{semilinearODE4x4}. Although the interest of this article does not lie in the well-posedness of the linear IBVP \eqref{linear_IBVP}, it is worth to briefly sketch the proof and refer the reader to \cite{IguLan21} for more details.
The first step is to prove an \textit{a priori} $L^2$ estimate taking advantage of the existence of a Kreiss symmetrizer provided in Lemma \ref{existenceKS}. 
As we explained in the previous sections,
this yields dissipativity on the boundary conditions and thanks to the good signs in the energy estimate we can get a control not only for $\|u(t)\|_{L^2(\mathcal{E}_+)}$ itself but also for the trace term $|u_{|_{x=
r
,l}}|_{L^2(0,t)}$. 
Next, one needs to generalise $L^2$ estimates to higher-order Sobolev spaces by employing commutator and Moser-type estimates. 
Finally, following classical arguments (see for instance \cite{benzoni-serre2007, metivier2001,metivier2012small}) the existence and uniqueness of the solution $u\in \mathbb{W}^m(T)$ is obtained from the \textit{a priori} estimates and the compatibility conditions.

\subsection{Nonlinear estimates}\label{subsection_nonlinear}
Let us state here some Moser-type nonlinear estimates (see for instance \cite{alinhac2012operateurs}) that we will use later in the analysis of Subsection \ref{subsection_IBVP}. We denote by $[k]$ the integer part of $k\in \mathbb{R_+}.$
\begin{lemma}\label{lemma_moser_sobolev}
		Let $\mathcal{U}$ be an open set in $\mathbb{R}^N$ and let $F \in C^\infty(\mathcal{U})$ be a function such that $F(0) = 0$. For $m\in\mathbb{N}$, if 
		$u \in H^m(0,T)$ takes values in a compact and convex set $\mathcal{K} \subsetneq \mathcal{U}$, then
		$$
		|F(u)|_{H^m(0,T)} \leq C_F(|u|_{W^{[m/2],\infty}(0,T)}) |u|_{H^m(0,T)}.
		$$
		Moreover, if $u \in H^m(0,T)$ and $v \in H^m(0,T)$ with $m\geq 1$ take values in $\mathcal{K}$, we have 
		$$
		|F(u) - F(v)|_{H^m(0,T)} \leq C_F (|u, v|_{H^m(0,T)}) 
		|u -v|_{H^m(0,T)}.
		$$
\end{lemma}	
	
\begin{lemma}[see \cite{IguLan21,metivier2012small}]\label{lemma_moser}
		Let $\mathcal{U}$ be an open set in $\mathbb{R}^N$ and let $F \in C^\infty(\mathcal{U})$ be a function such that $F(0) = 0$. For $m\in\mathbb{N}$, if 
		$u \in \mathbb{W}^m(T)$ takes values in a compact and convex set $\mathcal{K} \subsetneq \mathcal{U}$, then 
  for all $t \in [0, T]$:
		$$
		\vertiii{F(u)(t)}_m \leq C_F(\|u\|_{W^{[m/2],\infty}(\Omega_T)}) \vertiii{u(t)}_m.
		$$
		Moreover, if $u \in \mathbb{W}^m (T)$ and $v \in \mathbb{W}^m(T)$ with $m\geq 1$ take values in $\mathcal{K}$, we have 
		$$	\vertiii{(F(u) - F(v))(t)}_m \leq C_F(\vertiii{u(t),v(t)}_m) \vertiii{(u -v)(t)}_m.		$$
\end{lemma}

\begin{remark}
We use these nonlinear estimates because in the standard Moser nonlinear estimates $$\|F(u)\|_{H^m(D)}\leq C_F(\|u\|_{L^\infty(D)})\|u\|_{H^m(D)} \qquad \mbox{with}\quad D=\Omega_T \mbox{ or } (0,T),$$
the constant $C_F$ is time-dependent and blows-up as $T\rightarrow 0$. Since our goal is to use a contraction argument for the existence of the solution in which we will consider a small existence time $T$, we need nonlinear estimates with time-independent constants as the ones derived in Lemma \ref{lemma_moser_sobolev} and Lemma \ref{lemma_moser}. We refer to \cite{bresch2019waves,metivier2001} for sharp nonlinear estimates that provide blow-up criteria, in which the interest of this work does not lie.
\end{remark}

\subsection{Estimates for the ODE}\label{subsection_ODE}
We remark the fact that the boundary condition in the initial boundary value problem \eqref{IBVP4x4}--\eqref{semilinearODE4x4} is not a given information but it is a semi-linear boundary condition given by an ODE.
This subsection is devoted to establish Sobolev estimates for the solution to
\begin{equation}\label{general_ODE}
\dot{G} (t)= \Theta(G(t), u_{|_{x=
r
}}(t) ), \quad
G(0)=G_0,
\end{equation}
where $G(t) =(G_1(t),\cdots,G_N(t))^T$ is a $N$-dimensional function,  $u_{|_{x=r}}(t) = $ $((u_{1})_{|_{x=r}}(t),\\ \cdots, (u_{M})_{|_{x=r}}(t))^T$ is a given $M$-dimensional function and $\Theta = (\Theta_1,\cdots, \Theta_N)^T$ is a nonlinear smooth function. We construct a successive sequence of approximation solution $\{G^n\}_{n \in \mathbb{N} }$ to the Cauchy problem \eqref{general_ODE} defined by  
\begin{equation}\label{equation_Gn}
\dot{G}^{n+1} (t)= \Theta(G^n (t), {u^n}_{|_{x=
r
}}(t) ),\quad
G^{n+1}(0)=G_0.
\end{equation}
Some high-order estimates on the sequence $\{G^n\}_{n \in \mathbb{N} }$ are stated in the following proposition.
\begin{proposition}\label{propo_ode}
Let $\mathcal{G} \times \mathcal{U}_r$ be an open set in $\mathbb{R}^N \times \mathbb{R}^M$, representing a phase space of $(G, u_{|_{x=r}})$ and let $\Theta \in C^{\infty} (\mathcal{G} \times \mathcal{U}_r)$. 
Given $m\geq1$ and $T>0$, assume that $\{G^{n}\}_{n \in \mathbb{N} }\in H^{m+1}(0,T)$ and $\{ {u^n}_{|_{x=r}}\}_{n \in \mathbb{N} }\in H^m(0,T)$ satisfy \eqref{equation_Gn} and that they take values in compact and convex sets of $\mathcal{G}$ and $\mathcal{U}_r$ respectively. Moreover, assume that $(G^{n})^{(k)}(0)$ for $k=0,...,m$ and $(\partial_t^k{u^n})(0,r)$ for $k=0,...,m-1$ are independent of $n$ and that there exists $K_0>0$ such that
$$\sum\limits_{k=0}^{m} |(G^{n})^{(k)}(0)|,
\quad 
\sum\limits_{k=0}^{m-1} |(\partial^k_t{u^n})(0,r)| \leq K_0.$$
Then, we have \begin{align}\label{estimate_G_T}
|G^{n+1}|_{H^{m}(0,T)} 
&\leq \sqrt{T}C(K_0)+ TC_{\Theta}( K_0,|G^n, {u^n}_{|_{x=r}}|_{H^m(0,T)})
,\\[2pt]
\label{estimate_G}
|G^{n+1}|_{H^{m+1}(0,T)} 
&\leq \sqrt{T}C(K_0)+ (T+1)C_{\Theta}( K_0,|G^n, {u^n}_{|_{x=r}}|_{H^m(0,T)})
,
\end{align}
and
\begin{equation}
\label{estimate_G_diff}
\begin{aligned}
    |G^{n+1} - G^{n}|_{H^{m}(0,T)}  \leq & \
TC_{\Theta}(|G^n, G^{n-1}|_{H^m(0,T)}, |{u^n}_{|_{x=r}},{u^{n-1}}_{|_{x=r}}|_{H^{m}(0,T)})\\
 &\times (|G^{n} - G^{n-1}|_{H^{m}(0,T)} + |({u^{n}}-{u^{n-1}})_{|_{x=r}}|_{H^{m}(0,T)}).
\end{aligned}
\end{equation}

\end{proposition}
\begin{proof}
We divide the proof into two steps, one for each estimate.\\
\textbf{Step 1:} Let us first write the derivative of $G^{n+1}(t)$ of order $0\leq k \leq m$ as 
\begin{equation}\label{integral_X_n+1}
\begin{aligned}
  (G^{n+1})^{(k)} (t) & =  (G^{n+1})^{(k)}(0) + \int_{0}^{t} (\dot{G}^{n+1})^{(k)}(s) \ ds \\
   & = (G^{n+1})^{(k)}(0) + \int_{0}^{t}  \partial_s^k\Theta(G^{n}(s), {u^n}_{|_{x=
   r
   }} (s) ) \ ds ,
\end{aligned}
\end{equation}
where we have used the iterative ODE \eqref{general_ODE}.
Taking the sum over $k$ and using \eqref{equation_Gn} yield
\begin{equation}\label{estimate_Hm}
\begin{aligned}
&|G^{n+1}|_{H^{m+1}(0,T)}^2 
=|G^{n+1}|_{H^m(0,T)}^2+ |(\dot{G}^{n+1})^{(m)}|_{L^2(0,T)}^2 \\
& = \sum_{k=0}^{m}\left|(G^{n+1})^{(k)}(0) + \int_{0}^{t} \partial_s^k\Theta (G^n (s), {u^n}_{|_{x=
r
}}(s) )  ds \right|^2_{L^{2} (0, T)} + |\partial_t^m\Theta(G^n,{u^n}_{|_{x=
r
}} )|_{L^2(0,T)}^2\\
&\leq 2T \sum_{k=0}^{m} |(G^{n+1})^{(k)}(0)|^2 \!+\! 2\sum_{k=0}^{m} \left|\sqrt{t}\ | \partial_s^k\Theta (G^n , {u^n}_{|_{x=r}} )|_{L^2(0,t)} \right|^2_{L^{2} (0, T)} \\
&\qquad  
+ |\Theta(G^n,{u^n}_{|_{x=
r
}} )|_{H^m(0,T)}^2\\
& \leq T C(K_0) + (T^2+1) \left| \Theta (G^n , {u^n}_{|_{x=
r
}} ) \right|^2_{H^{m} (0, T)}.\\[2pt]
\end{aligned}
\end{equation}
Let us take any point $(G^*,u^*)\in \mathcal{G}\times \mathcal{U}_r$ and define $$\Theta_0(G, u_{|_{x=r}})= \Theta(G+G^*,  u_{|_{x=r}} + u^*) - \Theta(G^*, u^*).$$
Then, $\Theta_0\in C^\infty(\mathcal{G}\times \mathcal{U}_r)$ with $\Theta_0(0,0)=0$ and we have 
\begin{equation*}
    \begin{aligned}
        |\Theta(G^n, {u^n}_{|_{x=r}})|_{H^m(0,T)}&= |\Theta_0(G^n-G^*,{u^n}_{|_{x=r}}-u^*) + \Theta(G^*,u^*)|_{H^m(0,T)}\\[2pt]&\leq |\Theta_0(G^n-G^*,{u^n}_{|_{x=r}}-u^*)|_{H^m(0,T)} + |\Theta(G^*, u^*)|\sqrt{T}.
    \end{aligned}
\end{equation*}
The first estimate in Lemma \ref{lemma_moser_sobolev} gives 
\begin{equation*}
\begin{aligned}
 	&\left| \Theta_0(G^n -G^*, {u^n}_{|_{x=r}} -u^*) \right|_{H^{m} (0, T)} \\&
	\leq  \,C_{\Theta}(|G^n -G^*,  {u^n}_{|_{x=r}} -u^* |_{W^{[m/2],\infty}(0,T)} ) (|G^n -G^* |_{H^m(0,T)} + | {u^n}_{|_{x=r}} -u^*|_{H^m(0,T)}).\\[2pt]
\end{aligned}
\end{equation*}
By means of \eqref{integral_X_n+1} and using that $[m/2] + 1 \leq m$, we obtain
\begin{equation}\label{estimat_G^n-Winfty}
\begin{aligned}
|G^n -G^*|_{W^{[m/2],\infty}(0,T)}
 &\leq \sum_{k=0}^{[m/2]}| (G^n -G^*)^{(k)}(0)| + \sqrt{T} |G^n -G^*|_{H^{[m/2]+1}(0,T)}
\\&\leq C(K_0) +  \sqrt{T} |G^n -G^*|_{H^{m}(0,T)},
\end{aligned}
\end{equation}
and, analogously,
\begin{equation*}
\begin{aligned}
|{u^n}_{|_{x=r}} -u^*  |_{W^{[m/2],\infty}(0,T)}
& \leq \sum_{k=0}^{[m/2]}| (\partial_t^k
({u^n}_{|_{x=r}} -u^*) )(0,r)| + \sqrt{T} |{u^n}_{|_{x=r}} -u^*  |_{H^{[m/2]+1}(0,T)},\\
& \leq C(K_0) +  \sqrt{T} |{u^n}_{|_{x=r}} -u^*  |_{H^{m}(0,T)}.
\end{aligned}
\end{equation*}
Gathering all these estimates together yields
\begin{equation}\label{estimat_F_twofunc}
\begin{aligned}
\left| \Theta(G^n , {u^n}_{|_{x=
r
}} ) \right|_{H^{m} (0, T)} 
	&\leq  C_{\Theta}(K_0, |G^n-G^*, {u^n}_{|_{x=
	r
	}} -u^*|_{H^m(0,T)} )\\[2pt]
 & \leq C_{\Theta}(K_0, |G^n, {u^n}_{|_{x=
	r
	}}|_{H^m(0,T)} ),
\end{aligned}
\end{equation}
which, together with \eqref{estimate_Hm}, implies \eqref{estimate_G}. The $H^m$-estimate \eqref{estimate_G_T} is then straightforward.\\
\textbf{Step 2:}  
Using again \eqref{equation_Gn} and the fact that the initial conditions of $G^n$ and its derivatives are independent of $n$, we have for $0\leq k\leq m$
\begin{equation*}
    (G^{n+1})^{(k)} (t) - (G^{n})^{(k)} (t)    =  
  \int_{0}^{t}  \left( \partial_s^k\Theta(G^n (s), {u^n}_{|_{x=
r
 }} (s) )
- \partial_s^k\Theta(G^{n-1} (s), {u^{n-1}}_{|_{x=
r
}}(s) ) \right) \ ds.
\end{equation*}
Doing the same computation as in \eqref{estimate_Hm}, we obtain
\begin{equation}\label{estimate_diff_Gn}\begin{aligned}
   & |G^{n+1}-G^{n} |_{H^m(0,T)}^2 \\
&\leq 
 \sum_{k=0}^{m} \left|\sqrt{t} \ |( \partial_s^k\Theta(G^n , {u^n}_{|_{x=
r
}}
\!-\! \partial_s^k\Theta(G^{n-1}, {u^{n-1}}_{|_{x=
r
}}  )|_{L^2(0, t)} \right|^2_{L^{2} (0, T)}\\
& \leq 
 \frac{T^2}{2} \left| \Theta(G^{n}, {u^{n}}_{|_{x=
r
}} ) - \Theta(G^{n-1}, {u^{n-1}}_{|_{x=
r
}} ) \right|^2_{H^{m} (0, T)}.
\end{aligned}
\end{equation}
The second estimate in Lemma \ref{lemma_moser_sobolev} yields
\begin{equation}\label{estimate_ThetaG_diff}
\begin{aligned}
   &\left|\Theta(G^{n}, {u^{n}}_{|_{x=
   r
   }} ) \!- \!\Theta(G^{n-1}, {u^{n-1}}_{|_{x=
   r
   }} )\right|_{H^{m}(0,T)} 
\\& \leq   C_\Theta
(K_0, |G^{n-1}, {u^{n-1}}_{|_{x=
r
}}|_{H^m(0,T)} , |G^n, {u^n}_{|_{x=
r
}} |_{H^m(0,T)}) \\ &\qquad \qquad   \times (|G^{n}-G^{n-1}|_{H^{m}(0,T)} + |({u^{n}}-{u^{n-1}})_{|_{x=
 r
 }}|_{H^{m}(0,T)})
\end{aligned}
\end{equation}
and, by substituting this into \eqref{estimate_diff_Gn}, we obtain \eqref{estimate_G_diff}.
\end{proof}
\begin{remark}\label{remark_1} 
In Proposition \ref{propo_ode} we derived both $H^m$ and $H^{m+1}$-bounds \eqref{estimate_G_T}-\eqref{estimate_G} although one would look for the solution $G$ to \eqref{general_ODE} in the natural space $H^{m+1}(0,T)$.
 However, in the proof of the uniform boundedness of approximated solutions in Step 2 of Theorem \ref{maintheo}, while both $G^{n}$ and $u^{n}_{|_{x=0}}$ will belong to $H^{m}(0,T)$ by inductive hypothesis, the estimate \eqref{estimate_G} cannot directly guarantee the uniform bound of $G^{n+1}$ in $H^{m+1}(0,T)$ even for a small existence-time $T$. It is therefore crucial in our analysis to use first \eqref{estimate_G_T}, with a time-factor that allows to get the uniform bound in the $H^m$-regularity and, only afterwards, the expected $H^{m+1}$-regularity will be obtained using \eqref{estimate_G}.
\end{remark}

\subsection{Quasilinear hyperbolic IBVPs with semilinear boundary condition}\label{subsection_IBVP}
We now turn to consider the quasilinear hyperbolic initial boundary value problem
\begin{equation}
\begin{aligned}
\begin{cases}
\partial_t u + \mathcal{A}(u)\partial_x u = f(t,x) &\quad\mbox{in}\quad  \Omega_T,\\[5pt]
u(0)=u_0(x) &\quad\mbox{in}\quad  \mathcal{E}_+,\\[5pt]
\mathcal{M}_{
r
}
u_{|_{x=
r
}}=V(G(t)) &\quad\mbox{in}\quad  (0,T),
\\[5pt]
\mathcal{M}_l
u_{|_{x=l}}=g(t) &\quad\mbox{in}\quad  (0,T),
\end{cases}
\end{aligned}
\label{generalIBVP4x4}
\end{equation}
coupled with the evolution equation 
\begin{equation}
\begin{cases}
\dot{G}= \Theta(G, u_{|_{x=
r
}} ),\\[5pt]
G(0)=G_0.
\end{cases}
\label{generalsemilinearODE4x4}
\end{equation}
We require the following assumption:

\begin{assumption}\label{assumWP}
			Let $\mathcal{U}^-$ and $\mathcal{U}^+$ be two open sets in $\R^2$ such that $\mathcal{U}=\mathcal{U}^- \times \mathcal{U}^+$ represents a phase space of $u$. Let $\mathcal{U}^{-}_{r,l} \subset \mathcal{U}^-$  and $\mathcal{U}^+_{r,l} \subset \mathcal{U}^+$ be open sets such that $\mathcal{U}_{r,l}=\mathcal{U}^-_{r,l} \times \mathcal{U}^+_{r,l}$ represents a phase space of $u_{|_{x=r,l}}.$ Let $\mathcal{G}$ be an open set in $\R^2$ representing a phase space of $G$. The following properties are satisfied:
\begin{enumerate}
				\item[$(i)$] $\mathcal{A}\in C^\infty(\mathcal{U})$, $V\in C^\infty(\mathcal{G})$, $\Theta\in C^\infty(\mathcal{G} \times \mathcal{U}_r)$, $\det(\mathcal{M}_{
				r
				} \mathcal{M}_{
				r
				}^T)>0$, and  $  \det(\mathcal{M}_l \mathcal{M}_l^T)>0.$\\
				\item[$(ii)$]  Given $u=(u^-,u^+)^T\in \mathcal{U}$, $\mathcal{A}(u)= \mbox{diag}(-A^-(u^-),A^+(u^+))$ where $A^-(u^-)$ and $A^+(u^+)$ have eigenvalues $\pm\lambda_\pm(u^-)$  and $\pm\lambda_\pm(u^+)$ respectively, with $\lambda_\pm(u^-), \lambda_\pm(u^+) >0$. \\
				\item[$(iii)$]   For any $u_{|_{x={r,l}}}\in \mathcal{U}_{r,l}$, the $2\times 2$ Lopatinski\u{i} matrices  $\mathcal{L}_r(u_{|_{x=r}})$ and $\mathcal{L}_l(u_{|_{x=l}})$, defined as in Assumption \ref{assumWPlin}, are invertible.
\end{enumerate}
\end{assumption}

\paragraph{\textbf{Compatibility conditions}}
We write here the nonlinear version of the compatibility conditions already defined in Subsection \ref{subsection_linearsystem} for the linear problem. In order to guarantee the continuity of the solutions, on the edges $(t,x)=(0,r)$ and $(t,x)=(0,l)$ the initial data $u_0$, $G_0$ and the boundary data $g$ must satisfy
\begin{equation}
		\mathcal{M}_{
		r
		}
		{u_0}_{|_{x=
		r
		}}=V(G_0), \qquad \mathcal{M}_l
		{u_0}_{|_{x=l}}=g_0,
		\label{compatibility-0}\end{equation}
  with $g(0)=g_0.$ Analogously, defining $u_1=\partial_t u(0,x)$, $G_1=\dot{G}(0)$ and $g_1=\dot{g}(0)$, $C^1$-solutions must satisfy \eqref{compatibility-0} together with 
		\begin{equation}
		\mathcal{M}_{
		r
		}
	 {u_1}_{|_{x=
	 r
	 }}=D_V(G_0)G_1, \qquad \mathcal{M}_l
	    {u_1}_{|_{x=l}}=g_1, 
		\label{k1}\end{equation}
  where $D_V$ is the Jacobian matrix of $V$. We remark that $G_1$ can be written as well in terms of $u_0$ and $G_0$ using the ODE \eqref{generalsemilinearODE4x4}, namely
		$$G_1=\Theta(G_0, {u_0}_{|_{x=
		r
		}}).$$
  Hence, we can write \eqref{k1} under the form 
\begin{equation*}
		\mathcal{M}_{
		r
		} {u_1}_{|_{x=r}}=\mathcal{F}_1(G_0,{u_0}_{|_{x=
		r
		}}),
		\end{equation*}
		where  $\mathcal{F}_1(G_0,{u_0}_{|_{x=
		r
		}}) = D_V(G_0)\Theta(G_0,{u_0}_{|_{x=r}})$. More generally, let us define  $u_k=\partial_t^k u(0,x)$, $G_k= G^{(k)}(0)$ and $g_k= g^{(k)}(0)$ for $k\geq 1$. Then, smooth enough solutions must satisfy \eqref{compatibility-0} together with
\begin{equation*}
	\mathcal{M}_r {u_k}_{|_{x=r}}=\mathcal{F}_k(G_0,...,G_{k-1}, {u_0}_{|_{x=r}},..., {u_{k-1}}_{|_{x=
	r
	}}), \qquad \mathcal{M}_l {u_k}_{|_{x=l}}=g_k,
\label{compatibility-k}
\end{equation*}
where $\mathcal{F}_k$ is a smooth function of its arguments.\\ Let us now define $f_k=\partial_t^k f(0,x).$ On the one hand, using the evolution equation in \eqref{generalIBVP4x4} and applying an inductive argument, we can write ${u_k}$ as a function of the initial data $u_0$ and the source term $f$ only, namely
\begin{equation}
{u_{k}}=\mathcal{C}_{k}(u_0, f_0, ..., f_{k-1}) \qquad \mbox{for} \quad k\geq 1,
		\label{uk}
\end{equation}
where $\mathcal{C}_{k}(u_0, f_0, \!...\!, f_{k-1})$ is a smooth function of $u_0$, $\partial_x^{j+1} u_0$, and $\partial_x ^{k-1-j}f_j$ for $j\!=\!0,...,k-1$. 
On the other hand, using the ODE \eqref{generalsemilinearODE4x4}, \eqref{uk} and an inductive argument, we can write $G_k$ as a function of the data $G_0$, $u_0$ and $f$ only, that is 
\begin{equation}
G_1= \mathcal{B}_{1}(G_0,u_0
),\qquad G_{k}= \mathcal{B}_{k}(G_0,u_0, f_{0},  ..., f_{k-2}) \qquad \mbox{for} \quad k\geq 2,
\label{Gk}
\end{equation}
where $\mathcal{B}_{1}(G_0,u_0)$ is a smooth function of $G_0$, ${u_0}_{|_{x=r}}$ and $\mathcal{B}_{k}(G_0,u_0, f_{0},  ..., f_{k-2})$ is a smooth function of $G_0$, ${u_0}_{|_{x=
r
}}$, $(\partial_x^{j+1} u_0)_{|_{x=
r
}}$, $(\partial_x^{k-2-j} f_j)_{|_{x=
r
}}$ for $j=0,...,k-2.$
This permits us to introduce the next definition.
		
\begin{definition}\label{compatibility}
Let $m\geq 1$ be an integer. The data $u_0\in H^m(\mathcal{E}_+)$, $f\in H^m(\Omega_T)$, $g\in H^m(0,T)$ and $G_0\in\R$ of the initial boundary value problem \eqref{generalIBVP4x4}-\eqref{generalsemilinearODE4x4} satisfy the compatibility conditions up to order $m-1$ if 
\begin{equation*}
\begin{aligned}
&\mathcal{M}_{r} {u_0}_{|_{x=r}}=V(G_0), \\
&\mathcal{M}_{r} {u_k}_{|_{x=r}}=\mathcal{F}_k(G_0,...,G_{k-1}, {u_0}_{|_{x=r}},..., {u_{k-1}}_{|_{x=r}}) \quad\mbox{for}\quad k=1,...,m-1,
\label{compa}
\end{aligned}
\end{equation*}
and 
\begin{equation*}
    \mathcal{M}_l {u_k}_{|_{x=l}}=g_k \quad\mbox{for}\quad k=1,...,m-1.
\end{equation*}
\end{definition}	
We are now able to state a well-posedness result for an initial boundary value problem with a semi-linear boundary condition.
\begin{theorem}\label{maintheo}
			Let $m\geq 2$ be an integer. Assume that Assumption \ref{assumWP} holds and that $u_0\in H^m(
\mathcal{E}_+
			)$ takes values in $\mathcal{K}_0^-\times\mathcal{K}_0^+$ with $\mathcal{K}_0^-\subsetneq \mathcal{U}^-$ and $\mathcal{K}_0^+\subsetneq \mathcal{U}^+$ compact and convex sets, ${u_0}_{|_{x=r,l}}\in \mathcal{U}_{r,l}$ and $G_0\in \mathcal{G}$. Moreover, suppose that $u_0$, $f\in H^m(\Omega_T)$, $g\in H^m(0,T)$ and $G_0$ satisfy the compatibility conditions up to order $m-1$ in the sense of Definition \ref{compatibility}. Then,
			there exist $0<T_1\leq T$ and a unique solution $(u,G)$ to \eqref{generalIBVP4x4}-\eqref{generalsemilinearODE4x4} with $u\in\mathbb{W}^m(T_1)$ and $G\in H^{m+1}(0,T_1).$ Moreover $|u_{|_{x=
		r
			,l}}|_{m,T_1}$ is finite.
\end{theorem}
\begin{proof}
\textbf{Step 1: Choice of the iterative scheme.}
Let $\mathcal{K}_1^-,\mathcal{K}_1^+$ be two compact and convex sets in $\R^2$ such that $\mathcal{K}_0^-\times\mathcal{K}_0^+\Subset\mathcal{K}_1^-\times\mathcal{K}_1^+\Subset \mathcal{U}^-\times \mathcal{U}^+$ (compactly contained) and let $\mathcal{K}^{-}_{r,l,1}\times \mathcal{K}^+_{r,l,1}$ be a compact set in $\mathcal{U}_{r,l}$. Let $\mathcal{G}_1$ be a compact set in $\mathcal{G}$. Then, there exists a constant $c_0>0$ such that,  for any $u=(u^-,u^+)^T\in \mathcal{K}_1^-\times \mathcal{K}_1^+$ and $u_{|_{x={r,l}}}\in \mathcal{K}^{-}_{r,l,1}\times \mathcal{K}^+_{r,l,1}$,
\begin{align*}
\lambda_\pm(u^-)\geq c_0, \ \ \lambda_\pm(u^+) \geq c_0,  \ \ 
\|\mathcal{L}_r( u_{|_{x=r}})^{-1}\|_{\R^2\rightarrow \R^2}\leq \frac{1}{c_0}, \ \ \|\mathcal{L}_l( u_{|_{x=l}})^{-1}\|_{\R^2\rightarrow \R^2}\leq \frac{1}{c_0}.
\end{align*}
We construct a solution $(u,G)$, where $u$ takes values in $\mathcal{K}_1^-\times\mathcal{K}_1^+$, their traces $u_{|_{x={r,l}}}$ take values in $\mathcal{K}^{-}_{r,l,1}\times \mathcal{K}^+_{r,l,1}$ and $G$ takes values in $\mathcal{G}_1$. Indeed,  
there exists $\kappa_0>0$ such that, if
$\|u-u_0\|_{L^\infty(\mathcal{E}_+)}\leq \kappa_0$, we  have $u(x)\in \mathcal{K}_1^-\times\mathcal{K}_1^+$ for all $x\in 
\mathcal{E}_+.$ To do this, we use an iterative scheme argument. More precisely, we look for the solution as a limit of the sequence $(u^n, G^n)_{n\in\mathbb{N}}$, which solves
\begin{equation}
\begin{aligned}
\begin{cases}
\partial_t {u^{n+1}} + \mathcal{A}(u^n)\partial_x {u^{n+1}} = f(t,x)&\quad\mbox{in}\quad  \Omega_T,\\[5pt]
{u^{n+1}}(0)=u_0(x) &\quad\mbox{on}\quad   \mathcal{E}_+,\\[5pt]
\mathcal{M}_{r}{u^{n+1}}_{|_{x=r}}=V(G^{n+1}(t)) &\quad\mbox{on}\quad  (0,T),\\[5pt]
\mathcal{M}_l {u^{n+1}}_{|_{x=l}}=g(t) &\quad\mbox{on}\quad  (0,T),
\end{cases}
\end{aligned}
\label{iterativePDE}
\end{equation}
coupled with 
\begin{equation}
\begin{cases}
\dot{G}^{n+1}= \Theta(G^n, {u^n}_{|_{x=
r
}} ),\\[5pt]
G^{n+1}(0)=G_0.
\end{cases}
\label{iterativeODE}
\end{equation}
We choose the first iterate $(u^0, G^0)$ with a function $u^0\in H^{m+1}(\R \times \mathcal{E}_+)$ such that $(\partial_t^k u^0)(0,x)= u_k$  for $0\leq k\leq m$ with $u_k$ defined by \eqref{uk} and a function $G^0\in H^{m+1}(0,T)$ such that $(G^0)^{(k)}(0)=G_k$ with $G_k$ defined by \eqref{Gk}. The compatibility conditions are then satisfied by the data $u_0$ and $G_0$ for the linear initial boundary value problem for the unknown $({u^{n+1}}, G^{n+1})$. 
By construction, both quantities $$\vertiii{u^n(0)}_m=\sum\limits_{j=0}^m \|(\partial_t^ju^n)(0,\cdot)\|_{H^{m-j}(
\mathcal{E}_+
)}=\sum\limits_{j=0}^m \|u_j\|_{H^{m-j}(
 \mathcal{E}_+
)}, \qquad \sum\limits_{j=0}^m|(G^{n})^{(j)}(0)|=\sum\limits_{j=0}^m|G_j|,$$
are independent of $n$. Moreover, there exists $K_0>0$ such that
\begin{equation*}
\begin{aligned}
		&\frac{1}{c_0}, \ \vertiii{u^n(0)}_m+ |g|_{H^m(0,T)} + |f_{|_{x=
		r
		,l}}|_{m-1, T} + \int_{0}^{T} \vertiii{f(t)}_{m} dt,\\[2pt]& \|\mathcal{A}\|_{L^\infty(\mathcal{K}_1^-\times\mathcal{K}_1^+)}, \ \|\mathcal{M}_{
		r
		}\|_{\R^4\rightarrow \R^2},\|\mathcal{M}_l\|_{\R^4\rightarrow \R^2}, \ \sum\limits_{j=0}^m|(G^{n})^{(j)}(0)|\leq K_0.
\end{aligned}
\end{equation*}
\textbf{Step 2: High-norm boundedness.}
We want to show that the sequence $(u^n, G^n)_{n\in\mathbb{N}}$ is bounded in  $\mathbb{W}^{m}(T_1)\times H^{m+1}(0,T_1)$. We claim that for $M>0$ sufficiently large (to be determined later) and $0<T_1\leq T$ sufficently small  we have for all $n\in\mathbb{N}$:
\begin{equation}
\begin{cases}
\|u^n\|_{\mathbb{W}^m(T_1)}+ |{u^n}_{|_{x=
r
,l}}|_{m,T_1} + |G^n|_{H^{m+1}(0,T_1)}\leq M,\\[5pt]
\|u^n(t,\cdot)-u_0\|_{L^\infty(\mathcal{E}_+)}\leq \kappa_0 \quad\mbox{for}\ 0\leq t\leq T_1.
\end{cases}
\label{boundsequence}
\end{equation}
Let us first prove by an induction argument that for all $n\in \mathbb{N}$
\begin{equation}
\begin{cases}
\|u^n\|_{\mathbb{W}^m(T_1)}+ |{u^n}_{|_{x=
r
,l}}|_{m,T_1} + |G^n|_{H^{m}(0,T_1)}\leq \widetilde{M},\\[5pt]
\|u^n(t,\cdot)-u_0\|_{L^\infty(\mathcal{E}_+)}\leq \kappa_0 \quad\mbox{for}\ 0\leq t\leq T_1.
\end{cases}
\label{boundsequence_tilde}
\end{equation}
For the first iterate $n=0$ we have 
\begin{equation*}
\|u^0\|_{\mathbb{W}^m(T_1)}+ |{u^0}_{|_{x=
r
,l}}|_{m,T_1}+ |G^0|_{H^{m}(0,T_1)}\leq C(K_0),
\end{equation*}
for some constant $C(K_0)>0$ depending on $K_0$. 
Hence the first bound in \eqref{boundsequence_tilde} holds choosing $\widetilde{M}\geq C(K_0)$. Moreover, 
\begin{equation*}
			\|u^{0}(t,\cdot)-u_0\|_{L^\infty(\mathcal{E}_+)}\leq T_1 C\|u^{0}\|_{\mathbb{W}^2(T_1)}\leq T_1 C \widetilde{M},
\end{equation*}
and the second bound in \eqref{boundsequence_tilde} holds for $T_1\leq 
\frac{\kappa_0}{C \widetilde{M}}
$. We show now that \eqref{boundsequence_tilde} holds at step $n+1$ if it holds at step $n$. By interpolation, we have 
$$\|{u^{n}}\|^2_{W^{1,\infty}(\Omega_{T_1})}\leq C
\|{u^{n}}\|_{\mathbb{W}^{m-1}(T_1)} \|{u^{n}}\|_{\mathbb{W}^{m}(T_1)} \leq C(\widetilde{M}),$$
for some constant $C(\widetilde{M})>0$. By Theorem \ref{thm_Iguchi_lannes}, there exists a unique solution ${u^{n+1}}\in\mathbb{W}^m(T_1)$ to the initial boundary value problem \eqref{iterativePDE}. In addition, taking the supremum of \eqref{estimate_thm1} over $[0,T_1]$, the following estimate holds 
\begin{equation}\label{u_trace_n+1}
\begin{aligned}
	\|{u^{n+1}}\|_{\mathbb{W}^m (T_1)}+ |{u^{n+1}}_{|_{x=
	r
	,l}}|_{m,T_1} 
	\leq  C(K_0)e^{C(\widetilde{M})T_1} \left(1 +  |V(G^{n+1})|_{H^m(0,T_1)}\right).
\end{aligned}
\end{equation}
Arguing as in Step 1 of the proof of Proposition \ref{propo_ode} and using the first estimate in Lemma \ref{lemma_moser_sobolev} yield
\begin{equation*}
\begin{aligned}
|V(G^{n+1})|_{H^{m}(0,T_1)}
& \leq C_V (K_0, \sqrt{T_1} |G^{n+1}|_{H^{m}(0,T_1)})
|G^{n+1}|_{H^m(0,T_1)}
\end{aligned}
\end{equation*} 
while the bound \eqref{estimate_G_T} together with the inductive hypothesis \eqref{boundsequence_tilde} at step $n$ gives 
\begin{equation*}
|V(G^{n+1})|_{H^{m}(0,T_1)}\leq \sqrt{T_1}C_{V,\Theta}(K_0,T_1, \widetilde{M}).
\end{equation*}
Then, by choosing $T_1$ sufficiently small, we obtain  
\begin{equation*}
	\|{u^{n+1}}\|_{\mathbb{W}^m (T_1)}+ |{u^{n+1}}_{|_{x=
	r
	,l}}|_{m,T_1}  + |G^{n+1}|_{H^m(0,T_1)}
	\leq  2C(K_0)
\end{equation*}
and the first uniform bound in \eqref{boundsequence_tilde} is proved for all $n\in \mathbb{N}$ after setting $\widetilde{M}=2C(K_0)$. Moreover, 
\begin{equation*}
			\|{u^{n+1}}(t,\cdot)-u_0\|_{L^\infty(\mathcal{E}_+)}\leq T_1 \ C\|{u^{n+1}}\|_{\mathbb{W}^2(T_1)}\leq T_1 C \widetilde{M}\leq \kappa_0,
\end{equation*}
and the second uniform bound in \eqref{boundsequence_tilde} is proved for all $n\in\N$.\\
Now, in order to improve the regularity for $G^{n}$ to $H^{m+1}$ and prove the uniform bound \eqref{boundsequence}, we resort to \eqref{estimate_G}. Indeed, using  \eqref{boundsequence_tilde} yields 
\begin{equation*}\begin{aligned}
    |G^{n+1}|_{H^{m+1}(0,T_1)}&\leq \sqrt{T_1}C(K_0)+ (T_1+1)C_{\Theta}( K_0,|G^n, {u^n}_{|_{x=r}}|_{H^m(0,T_1)})\\
    &\leq \sqrt{T_1}C(K_0)+ (T_1+1)C_{\Theta}(K_0,\widetilde{M})
    \end{aligned}
\end{equation*}
and, for $T_1$ sufficiently small, 
$$|G^{n+1}|_{H^{m+1}(0,T_1)}\leq C_{\Theta}( K_0,\widetilde{M}).$$ Thus, we obtain
\begin{equation*}
\begin{aligned}
 \|{u^{n}}\|_{\mathbb{W}^m (T_1)}+ |{u^{n}}_{|_{x=r,l}}|_{m,T_1} + |G^{n}|_{H^{m+1}(0,T_1)}
 \leq \widetilde{M} + C_{\Theta}(K_0, \widetilde{M}),
\end{aligned}
\end{equation*}which proves \eqref{boundsequence} for all $n\in\mathbb{N}$ with $M=\widetilde{M} + C_{\Theta}(K_0, \widetilde{M})$.\\

\noindent\textbf{Step 3: Low-norm convergence.}
We show that $(u^n, G^n)$ is a convergent sequence in the $\mathbb{W}^{m-1}(T_1)\times H^{m-1}(0,T_1)$-norm. The initial boundary value problem for the difference ${u^{n+1}}-u^n$ reads
\begin{equation}
\begin{aligned}
\begin{cases}
			\partial_t ({u^{n+1}}-u^n) + \mathcal{A}(u^n)\partial_x ({u^{n+1}}-u^n) = f^n&\quad\mbox{in}\quad  \Omega_T,\\[5pt]
			({u^{n+1}}-u^n)
			(0)
			=0&\quad\mbox{on}\quad  {\mathcal{E}_+},\\[5pt]
			\mathcal{M}_r ({u^{n+1}}-u^n)_{|_{x=r}}=V(G^{n+1}(t))-V(G^{n}(t)) &\quad\mbox{on}\quad  (0,T),\\[5pt]
			\mathcal{M}_l ({u^{n+1}}-u^n)_{|_{x=l}}=0 &\quad\mbox{on}\quad  (0,T),
\end{cases}
\end{aligned}
\label{iterativePDEdiff}
\end{equation}
with source term  $f^n=-\left(\mathcal{A}(u^n)-\mathcal{A}({u^{n-1}})\right)\partial_x u^n.$
Applying Theorem \ref{thm_Iguchi_lannes} to \eqref{iterativePDEdiff} and taking the supremum over $[0,T_1]$,  we get
\begin{equation*}
\begin{aligned}
&\|{u^{n+1}}-u^n\|_{\mathbb{W}^{m-1} (T_1)}  + |({u^{n+1}}-u^n)_{|_{x=r,l}}|_{m-1, T_1} \\
&\leq   C(K_0)e^{C(M)T_1} \big(|V(G^{n+1})-V(G^{n})|_{H^{m-1}(0, T_1)}  + |f^n_{|_{x=r,l}}|_{m-2, T_1} + T_1\|f^n\|_{\mathbb{W}^{m-1} (T_1)}  \big). 
\end{aligned}
\end{equation*}
We estimate the right-hand side using Lemma \ref{lemma_appen_convergence} in Appendix and we get 
\begin{align*}
    &\|{u^{n+1}}-u^n\|_{\mathbb{W}^{m-1}(T_1)} + |({u^{n+1}}-u^n)_{|_{x=r,l}}|_{m-1,T_1}\leq C_{V, \Theta}(K_0, M)e^{C(M)T_1} T_1  \\
    &\times 
\big(\|{u^{n}}-{u^{n-1}}\|_{\mathbb{W}^{m-1}(T_1)}+ |({u^{n}}-{u^{n-1}})_{|_{x=r,l}}|_{m-1,T_1}+ |G^{n}-G^{n-1}|_{H^{m-1}(0,T_1)}\big).
\end{align*}
Using \eqref{estimate_G_diff} yields
\begin{equation*}
|G^{n+1} - G^{n}|_{H^{m-1}(0,T_1)}
\leq T_1 C_{\Theta}(M)
\left(|G^{n}-G^{n-1}|_{H^{m-1}(0,T_1)} +|({u^{n}}-{u^{n-1}})_{|_{x=r}}|_{m-1,T_1}\right), 
\end{equation*}
thus we obtain
\begin{equation*}
\begin{aligned}
&\|{u^{n+1}}-u^n\|_{\mathbb{W}^{m-1}(T_1)} +|({u^{n+1}}-u^n)_{|_{x=r,l}}|_{m-1,T_1}+ |G^{n+1}-G^{n}|_{H^{m-1}(0,T_1)}\\&\leq C(K_0, M)e^{C(M)T_1} T_1 \\&\quad \times \big(\|{u^{n}}-{u^{n-1}}\|_{\mathbb{W}^{m-1}(T_1)} + |({u^{n}}-{u^{n-1}})_{|_{x=r,l}}|_{m-1,T_1}+ |G^{n}-G^{n-1}|_{H^{m-1}(0,T_1)}\big).
\end{aligned}
\end{equation*}
Hence, by taking $T_1$ sufficiently small, we get 
\begin{equation*}
\begin{aligned}
& \|{u^{n+1}}-u^n\|_{\mathbb{W}^{m-1}(T_1)} + |({u^{n+1}}-u^n)_{|_{x=r,l}}|_{m-1,T_1}+ |G^{n+1}-G^{n}|_{H^{m-1}(0,T_1)}\\
& \leq \frac{1}{2} \left(\|{u^{n}}-{u^{n-1}}\|_{\mathbb{W}^{m-1}(T_1)}+ |({u^{n}}-{u^{n-1}})_{|_{x=r,l}}|_{m-1,T_1}+ |G^{n}-G^{n-1}|_{H^{m-1}(0,T_1)}\right). 
\end{aligned}
\end{equation*}
Thus, $(u^n,G^n)$ is a Cauchy sequence and converges in $\mathbb{W}^{m-1}(T_1) \times H^{m-1}(0,T_1)$  
to a limit $(u, G)$.\\
\textbf{Step 4: Regularity and uniqueness}. We have the following two interpolation inequalities
$$\|{u^{n+1}}-{u^{n}}\|^2_{W^{1,\infty}(\Omega_{T_1})}\leq C
\|{u^{n+1}}-{u^{n}}\|_{\mathbb{W}^{m-1}(T_1)} \|{u^{n+1}}-{u^{n}}\|_{\mathbb{W}^{m}(T_1)},$$
and
$$|G^{n+1}-G^{n}|^2_{H^{m}(0,T_1)}\leq C
|G^{n+1}-G^{n}|_{H^{m-1}(0,T_1)} |G^{n+1}-G^{n}|_{H^{m+1}(0,T_1)}.$$
From the uniform boundedness of $(u^n,G^n)$ in $\mathbb{W}^{m}(T_1) \times H^{m+1}(0,T_1)$ and the convergence of $(u^n,G^n)$ in  $\mathbb{W}^{m-1}(T_1) \times H^{m-1}(0,T_1)$,  
we can conclude that $(u^n,G^n)$ converges to $(u,G)$ in $\left(\mathbb{W}^{m-1}(T_1)\cap W^{1,\infty}(\Omega_{T_1})\right) \times H^{m}(0,T_1)$ and $(u, G)$ is a solution to
\eqref{generalIBVP4x4}-\eqref{generalsemilinearODE4x4}. By standard compactness arguments we have $$\|u\|_{\mathbb{W}^m(T_1)} + |u_{|_{x=r,l}}|_{m,T_1}
+ |G|_{H^{m+1}(0,T_1)}\leq M,$$ and the uniqueness of the solution is obtained via a standard energy estimate argument applied to the initial boundary value problem satisfied by the difference of two solutions. 
\end{proof}

\subsection{Well-posedness of the transmission problem across the structure side-walls}\label{subsection_well-posedness}
As a direct consequence of Theorem \ref{maintheo}, we can now prove Theorem \ref{theomainresult}, which states the well-posedness result of the transmission problem \eqref{transproblemcomplete}-\eqref{initial-cond} describing the interaction between the waves and the partially-immersed structure in the OWC device. Let us recall the statement below in Theorem \ref{thm_owc} to be more complete. Before giving its proof, we need to introduce the following assumption on the initial data.

\begin{assumption}\label{assOWC}There exists $c_0>0$ such that the initial data $(\zeta_0,q_0)$ satisfy:
            \begin{equation*}
                 g(h_0+\zeta_0(x))-\dfrac{q_0^2(x)}{(h_0+\zeta_0(x))^2}\geq c_0 \qquad \forall  x\in \mathcal{E}.
            \end{equation*}
\end{assumption}
\begin{theorem}\label{thm_owc}
Let $m\geq 2$ be an integer and $(\zeta_0, q_0)\in H^m(\mathcal{E})$ be such that  Assumption \ref{assOWC} holds. Moreover, suppose that $(\zeta_0,q_0)$, $(q_{i,0}, P_\mathrm{ch,0})\in\R^2$ and  $\zeta_{\rm ent}\in H^m(0,T)$ satisfy the compatibility conditions in Definition \ref{compatibility} up to order $m-1$. Then there exists $0<T_1\leq T$ and unique solution $(\zeta, q, q_i, P_\mathrm{ch})$ to
\eqref{transproblemcomplete}-\eqref{initial-cond} with $(\zeta, q)\in \mathbb{W}^m(T_1)$ and $(q_i,P_\mathrm{ch})\in H^{m+1}(0,T_1)$, where
$\mathbb{W}^m(T_1)$ denotes the same space as in \eqref{defWm} but defined in the spatial domain $\mathcal{E}$. Moreover, $|(\zeta, q)_{|_{x=\pm r, \pm l}}|_{m,T_1}$ is finite.
\end{theorem}

\begin{proof}
In order to apply Theorem \ref{maintheo}, we need to show that the conditions $(i)$-$(iii)$ in Assumption \ref{assumWP} are satisfied.
The  condition $(i)$ holds from the definition of $\mathcal{A}$, $V$, $\Theta$ in \eqref{IBVP4x4}-\eqref{semilinearODE4x4} and the boundary matrices 
$$\mathcal{M}_r=
\left(\begin{matrix}
			0& -1& 0&1 \\[5pt]
		0&\frac{1}{2}&0&\frac{1}{2}
				\end{matrix}
				\right), \qquad 
			\mathcal{M}_l=	\left(\begin{matrix}
			1& 0& 0&0 \\[5pt]
		0&0&0&1
				\end{matrix}
				\right).
$$
After remarking that the eigenvalues of $A(u)$ in \eqref{transpb2x2} are 
\begin{equation*}
    \pm \lambda_{\pm}(u)=\dfrac{q}{h_0+\zeta} \pm \sqrt{g(h_0+\zeta)},
\end{equation*}
Assumption \ref{assOWC} implies the condition $(ii)$. Therefore we only need to verify the condition $(iii)$. Let us recall that the unit eigenvectors of $A(u)$ associated with the eigenvalues $\pm\lambda_\pm(u)$ are respectively 
$$
e_\pm(u)=\frac{1}{\sqrt{1+ |\lambda_\pm(u)|^2}}\left(1, \pm\lambda_\pm(u) \right)^T.
$$
From the definition of the Lopatinski\u{i} matrices and writing $u$ as $u^-$ in $\mathcal{E}_-$ and as $u^+$ in $\mathcal{E}_+$, we obtain that the Lopatinski\u{i} matrices for the $4\times 4$ system \eqref{IBVP4x4} associated with \eqref{transproblemcomplete}-\eqref{equation_P_q} are
$$\mathcal{L}_r(u_{
|_{x=r}
})=\left(\begin{matrix}
 			\frac{\lambda_-\left({u^-}_{
 			|_{x=r}
 			}
 		\right)}{\sqrt{1+\left|\lambda_-\left({u^-}_{
 		|_{x=r}
 		}\right)\right|^2}}& \frac{\lambda_+\left({u^+}_{
 		|_{x=r}
 		}\right)}{\sqrt{1+\left|\lambda_+\left({u^+}_{
 		|_{x=r}
 		}\right)\right|^2}}  \\[13pt]
	\frac{-\lambda_-\left({u^-}_{
	 |_{x=r}
	}\right)}{2\sqrt{1+\left|\lambda_-\left({u^-}_{
	 |_{x=r}	}\right)\right|^2}}&\frac{\lambda_+\left({u^+}_{
	|_{x=r}
	}\right)}{2\sqrt{1+\left|\lambda_+\left({u^+}_{
	 |_{x=r}
	}\right)\right|^2}}
				\end{matrix}
				\right)$$
    and			$$\mathcal{L}_l(u_{|_{x=l}})=\left(\begin{matrix}
				\frac{1}{\sqrt{1+\left|\lambda_+\left({u^-}_{|_{x=l}}\right)\right|^2}}& 0 \\[10pt]
	0&	\frac{-\lambda_-\left({u^+}_{|_{x=l}}\right)}{\sqrt{1+\left|\lambda_-\left({u^+}_{|_{x=l}}\right)\right|^2}}
				\end{matrix}
				\right).\vspace{1em}$$ 
				From Assumption \ref{assOWC} we know that $\mathcal{L}_r(u_{
		|_{x=r}
				})$ and $\mathcal{L}_l(u_{|_{x=l}})$ are invertible, yielding the condition $(iii)$.
Then, the well-posedness result follows from Theorem \ref{maintheo}.
\end{proof}

\begin{appendices}

\renewcommand\sectionname{Appendix}
\section{Some technical estimates}
In this appendix, we prove some technical estimates that we have omitted in the proof of Theorem \ref{maintheo} for the sake of readability.

\begin{lemma}\label{lemma_appen_convergence}
Let $m\geq 2$ be an integer and assume that $u^n$, ${u^n}_{|_{x=r}}$ and $ G^n$ take values in compact and convex sets respectively in $\mathcal{U}$, $\mathcal{U}_r$ and $\mathcal{G}$ for all $n\in \mathbb{N}$. Suppose that Assumption \ref{assumWP} and the uniform bound \eqref{boundsequence} hold.  Then:
\begin{align} 
    	&| \left(\mathcal{A}(u^n)-\mathcal{A}({u^{n-1}}))\partial_x u^n\right)_{|_{x=r,l}}|_{m-2,T_1}\leq T_1 C(M) |(u^n-{u^{n-1}})_{|_{x=r,l}}|_{m-1,T_1}, \label{traceboundAn+1-n}\\[5pt]
&\|\mathcal{A}(u^n)-\mathcal{A}({u^{n-1}}))\partial_x u^n\|_{\mathbb{W}^{m-1}(T_1)} \leq  C(M)\|{u^{n}}-{u^{n-1}}\|_{\mathbb{W}^{m-1}(T_1),} \label{boundAn+1-n}\\[7pt]
  &  |V(G^{n+1})-V(G^{n})|_{H^{m-1}(0,T_1)} \nonumber \\ &\leq  T_1 C_{V, \Theta} (M) (|G^{n}-G^{n-1}|_{H^{m-1}(0,T_1)}+ |(u^n-{u^{n-1}})_{|_{x=r}}|_{m-1,T_1}). \label{boundVGn+1-n}
\end{align}
\end{lemma}
\begin{proof}
First, from Assumption \ref{assumWP} we have $\mathcal{A}\in C^\infty (\mathcal{U})$, then the second estimate in Lemma \ref{lemma_moser_sobolev} and \eqref{boundsequence} give
\begin{equation*}
\begin{aligned}
	& |((\mathcal{A}(u^n)-\mathcal{A}({u^{n-1}}))\partial_x u^n)_{|_{x=r, l}}|_{m-2,T_1} \\
	&\leq C\sum\limits_{|\alpha|+ |\beta|\leq m-2} |\partial^\alpha(\mathcal{A}({u^{n}})-\mathcal{A}({u^{n-1}}))_{|_{x=r, l}}|_{L^2(0,T_1)} \|(\partial^\beta\partial_x u^n)_{|_{x=r, l}}\|_{L^\infty(0,T_1)}\\
	&\leq C\sum\limits_{|\alpha|+ |\beta|\leq m-2} |\partial^\alpha(\mathcal{A}({u^{n}})-\mathcal{A}({u^{n-1}}))_{|_{x=r, l}}|_{L^2(0,T_1)}\|\partial^\beta u^n\|_{L^\infty(0,T_1;H^2(\mathcal{E}_+))}\\
	&\leq C(M) \sum\limits_{|\alpha|\leq m-2} |\partial^\alpha(u^n-{u^{n-1}})_{|_{x=r, l}}|_{L^2(0,T_1)}\| u^n\|_{\mathbb{W}^{m}(T_1)}\\[5pt]
	&\leq  C(M) |(u^n-{u^{n-1}})_{|_{x=r, l}}|_{m-2,T_1}
    \leq T_1 C(M) |(u^n-{u^{n-1}})_{|_{x=r, l}}|_{m-1,T_1}.
\end{aligned}
\end{equation*} 
Moreover, we have
\begin{equation*}
\begin{aligned}
 &\vertiii{((\mathcal{A}(u^n)-\mathcal{A}({u^{n-1}}))\partial_x u^n)(t)}_{m-1}
\\[2pt]&\leq C  \vertiii{\left(\mathcal{A}(u^n)-\mathcal{A}({u^{n-1}}) \right) (t)}_{m-1}  \vertiii{\d_x u^n (t)}_{m-1} \leq  C \vertiii{(u^n-{u^{n-1}}) (t)}_{m-1}  \vertiii{u^n (t)}_{m},\\[2pt]
\end{aligned}
\end{equation*}
where in the last inequality we have used the second estimate in Lemma \ref{lemma_moser}. Taking the supremum over $[0, T_1]$ and using \eqref{boundsequence} we get \eqref{boundAn+1-n}.\\
From Assumption \ref{assumWP}, we know that $V\in C^\infty(\mathcal{G})$ and $\Theta\in C^\infty(\mathcal{G}\times \mathcal{U}_r)$. Hence the second estimate in Lemma \ref{lemma_moser_sobolev} and \eqref{estimate_G_diff} yield
\begin{equation*}
\begin{aligned}
&|V(G^{n+1})\!-\!V(G^{n})|_{H^{m-1}(0,T_1)} \\\!&\leq  C_V ( |G^{n+1}, G^{n}|_{W^{[\frac{m-1}{2}], \infty}(0, T_1)}) |G^{n+1}-G^{n}|_{H^{m-1}(0, T_1)} \\
& \leq T_1  C_{V, \Theta} ( | G^{n}, {u^{n}}_{|_{x=r}}|_{W^{[\frac{m-1}{2}], \infty}(0, T_1)}, |G^{n+1},{u^{n+1}}_{|_{x=r}} |_{W^{[\frac{m-1}{2}], \infty}(0, T_1)}) 
\\ &\qquad \quad \times  (|G^{n}-G^{n-1}|_{H^{m-1}(0,T_1)}+ |(u^n-{u^{n-1}})_{|_{x=r}}|_{m-1,T_1}). \\
\end{aligned}
\end{equation*}
Using again \eqref{boundsequence}, we then prove \eqref{boundVGn+1-n}.
\end{proof}
	
\end{appendices}

\paragraph{\bf Acknowledgements.} 
The authors warmly thank David Lannes and Geoffrey Beck for their helpful comments and advises. They also thank the anonymous referees for their accurate remarks and comments that improved the article. The first author is supported by the Starting Grant project ``Analysis of moving incompressible fluid interfaces” (H2020-EU.1.1.-639227) operated by the European Research Council.
This material is also based upon work supported by the National Science Foundation under Grant No. DMS-1928930 while the second author participates in a program hosted by the Mathematical Sciences Research Institute in Berkeley, California, during the Spring 2021 semester.
The third author is supported by the European Unions Horizon 2020 research and innovation programme under the Marie Sklodowska-Curie grant agreement No.765579-ConFlex.\\

\bibliographystyle{siam}
\bibliography{mybib}

\begin{thebibliography}{10}

\bibitem{alinhac2012operateurs}
{\sc S.~Alinhac and P.~G{\'e}rard}, {\em Pseudo-differential operators and the
  Nash-Moser theorem}, vol.~82, Graduate Studies in Mathematics Series,
  American Mathematical Society, 2007.

\bibitem{babarit2018}
{\sc A.~Babarit}, {\em Ocean Wave Energy Conversion: Resource, Technologies and
  Performance}, Elsevier, 2018.

\bibitem{bastin-coron2016}
{\sc G.~Bastin and J.-M. Coron}, {\em Stability and Boundary Stabilization of
  1-{D} Hyperbolic Systems}, vol.~88 of Progress in Nonlinear Differential
  Equations and their Applications, Birkh\"{a}user/Springer, 2016.

\bibitem{lannesbeck2021boussinesq}
{\sc G.~Beck and D.~Lannes}, {\em Freely floating objects on a fluid governed
  by the {B}oussinesq equations}, Annales de l'Institut Henri Poincar\'{e} C.
  Analyse Non Lin\'{e}aire, 39 (2022), pp.~575--646.

\bibitem{benzoni-serre2007}
{\sc S.~Benzoni-Gavage and D.~Serre}, {\em Multi-dimensional hyperbolic partial
  differential equations: First-order Systems and Applications}, Oxford
  University Press on Demand, 2007.

\bibitem{bocchi2020floating}
{\sc E.~Bocchi}, {\em Floating structures in shallow water: local
  well-posedness in the axisymmetric case}, SIAM Journal on Mathematical
  Analysis, 52 (2020), pp.~306--339.

\bibitem{bocchi2020onthereturn}
\leavevmode\vrule height 2pt depth -1.6pt width 23pt, {\em On the return to
  equilibrium problem for axisymmetric floating structures in shallow water},
  Nonlinearity, 33 (2020), pp.~3594--3619.

\bibitem{bocchihevergara2021}
{\sc E.~Bocchi, J.~He, and G.~Vergara-Hermosilla}, {\em Modelling and
  simulation of a wave energy converter}, ESAIM: Proceedings and Surveys, 70
  (2021), pp.~68--83.

\bibitem{bresch2019waves}
{\sc D.~Bresch, D.~Lannes, and G.~M{\'e}tivier}, {\em Waves interacting with a
  partially immersed obstacle in the {B}oussinesq regime}, Analysis \& PDE, 14
  (2021), pp.~1085--1124.

\bibitem{coron_2007_book}
{\sc J.-M. Coron}, {\em Control and Nonlinearity}, vol.~136 of Mathematical
  Surveys and Monographs, American Mathematical Society, Providence, RI, 2007.

\bibitem{cummins1962impulse}
{\sc W.~E. Cummins}, {\em The impulse response function and ship motions},
  David Taylor Model Basin Washington DC,  (1962).

\bibitem{di1857note}
{\sc F.~Di~Bruno}, {\em Note sur une nouvelle formule de calcul
  diff{\'e}rentiel}, Quarterly Journal Pure Applied Mathematics, 1 (1857),
  pp.~359--360.

\bibitem{Dimakopoulos-Cooker-Bruce2017}
{\sc A.~S. Dimakopoulos, M.~J. Cooker, and T.~Bruce}, {\em The influence of
  scale on the air flow and pressure in the modelling of oscillating water
  column wave energy converters}, International Journal of Marine Energy, 19
  (2017), pp.~272--291.

\bibitem{evans1978}
{\sc D.~V. Evans}, {\em The oscillating water column wave-energy device}, IMA
  Journal of Applied Mathematics, 22 (1978), pp.~423--433.

\bibitem{evans1982}
\leavevmode\vrule height 2pt depth -1.6pt width 23pt, {\em Wave-power
  absorption by systems of oscillating surface pressure distributions}, Journal
  of Fluid Mechanics, 114 (1982), pp.~481--499.

\bibitem{evans1995}
{\sc D.~V. Evans and R.~Porter}, {\em Hydrodynamic characteristics of an
  oscillating water column device}, Applied Ocean Research, 17 (1995),
  pp.~155--164.

\bibitem{fanhelgato16air}
{\sc A.~F. Falc{\~a}o, J.~C. Henriques, and L.~M. Gato}, {\em Air turbine
  optimization for a bottom-standing oscillating-water-column wave energy
  converter}, Journal of Ocean Engineering and Marine Energy, 2 (2016),
  pp.~459--472.

\bibitem{IguLan21}
{\sc T.~Iguchi and D.~Lannes}, {\em Hyperbolic free boundary problems and
  applications to wave-structure interactions}, Indiana University Mathematics
  Journal, 70 (2021), pp.~353--464.

\bibitem{john1949}
{\sc F.~John}, {\em On the motion of floating bodies. {I}}, Communications on
  Pure and Applied Mathematics, 2 (1949), pp.~13--57.

\bibitem{john1950}
\leavevmode\vrule height 2pt depth -1.6pt width 23pt, {\em On the motion of
  floating bodies. {II}. {S}imple harmonic motions}, Communications on Pure and
  Applied Mathematics, 3 (1950), pp.~45--101.

\bibitem{lannes2017floating}
{\sc D.~Lannes}, {\em On the dynamics of floating structures}, Annals of PDE, 3
  (2017), pp.~11--81.

\bibitem{lannes2020modeling}
\leavevmode\vrule height 2pt depth -1.6pt width 23pt, {\em Modeling shallow
  water waves}, Nonlinearity, 33 (2020), pp.~1--57.

\bibitem{lannes2020generating}
{\sc D.~Lannes and L.~Weynans}, {\em Generating boundary conditions for a
  {B}oussinesq system}, Nonlinearity, 33 (2020), pp.~6868--6889.

\bibitem{li-yu1985}
{\sc T.~T. Li and W.~Yu}, {\em Boundary value problems for quasilinear
  hyperbolic systems}, vol.~V, Duke University Mathematics Series, Duke
  University, Mathematics Department, Durham, NC, 1985.

\bibitem{lopez2015}
{\sc I.~L{\'o}pez, B.~Pereiras, F.~Castro, and G.~Iglesias}, {\em Performance
  of {OWC} wave energy converters: influence of turbine damping and tidal
  variability}, International Journal of Energy Research, 39 (2015),
  pp.~472--483.

\bibitem{mai-sm-taka-tuc19}
{\sc D.~Maity, J.~San~Mart{\'\i}n, T.~Takahashi, and M.~Tucsnak}, {\em Analysis
  of a simplified model of rigid structure floating in a viscous fluid},
  Journal of Nonlinear Science, 29 (2019), pp.~1975--2020.

\bibitem{metivier2001}
{\sc G.~M{\'e}tivier}, {\em Stability of multidimensional shocks}, in Advances
  in the theory of shock waves, vol.~47, Progress Nonlinear Differential
  Equations Applied, Birkh{\"a}user Boston, Inc., Boston, MA, 2001,
  pp.~25--103.

\bibitem{metivier2012small}
\leavevmode\vrule height 2pt depth -1.6pt width 23pt, {\em Small Viscosity and
  Boundary Layer Methods: Theory, Stability Analysis, and Applications},
  Modeling and Simulation in Science, Engineering and Technology,
  Birkh{\"a}user Boston, Inc., Boston, MA, 2004.

\bibitem{pecher2017handbook}
{\sc A.~Pecher and J.~Peter~Kofoed}, {\em Handbook of Ocean Wave Energy},
  Springer Nature, 2017.

\bibitem{raghunathan1995wells}
{\sc S.~Raghunathan}, {\em The wells air turbine for wave energy conversion},
  Progress in Aerospace Sciences, 31 (1995), pp.~335--386.

\bibitem{Rashad-Califano-vanderSchaft_2020}
{\sc R.~Rashad, F.~Califano, A.~J. van~der Schaft, and S.~Stramigioli}, {\em
  Twenty years of distributed port-{H}amiltonian systems: a literature review},
  IMA Journal of Mathematical Control and Information, 37 (2020),
  pp.~1400--1422.

\bibitem{reza2013}
{\sc K.~Rezanejad, J.~Bhattacharjee, and C.~G. Soares}, {\em Stepped sea bottom
  effects on the efficiency of nearshore oscillating water column device},
  Ocean Engineering, 70 (2013), pp.~25--38.

\bibitem{reza2018}
{\sc K.~Rezanejad and C.~G. Soares}, {\em Enhancing the primary efficiency of
  an oscillating water column wave energy converter based on a dual-mass system
  analogy}, Renewable Energy, 123 (2018), pp.~730--747.

\bibitem{reza2017}
{\sc K.~Rezanejad, C.~G. Soares, I.~L{\'o}pez, and R.~Carballo}, {\em
  Experimental and numerical investigation of the hydrodynamic performance of
  an oscillating water column wave energy converter}, Renewable Energy, 106
  (2017), pp.~1--16.

\bibitem{vanderschaft-maschke_2002}
{\sc A.~J. van~der Schaft and B.~M. Maschke}, {\em Hamiltonian formulation of
  distributed-parameter systems with boundary energy flow}, Journal of Geometry
  and Physics, 42 (2002), pp.~166--194.

\bibitem{vergara-leugering-wang2021}
{\sc G.~Vergara-Hermosilla, G.~Leugering, and Y.~Wang}, {\em Boundary
  controllability of a system modelling a partially immersed obstacle}, ESAIM:
  Control, Optimisation and Calculus of Variations, 27 (2021).
\newblock Paper No. 80, 15 pp.

\bibitem{mati-tuc-vergara21}
{\sc G.~Vergara-Hermosilla, D.~Matignon, and M.~Tucsnak}, {\em Asymptotic
  behaviour of a system modelling rigid structures floating in a viscous
  fluid}, IFAC-PapersOnLine, 54 (2021), pp.~205--212.

\bibitem{wang2018nonlinear}
{\sc R.~Wang, D.~Ning, C.~Zhang, Q.~Zou, and Z.~Liu}, {\em Nonlinear and
  viscous effects on the hydrodynamic performance of a fixed {OWC} wave energy
  converter}, Coastal Engineering, 131 (2018), pp.~42--50.

\end{thebibliography}

\end{document}